\RequirePackage{ifthen}
\RequirePackage{ifpdf}
\newcommand{\driverOption}{}
\ifthenelse{\boolean{pdf}}{
  \renewcommand{\driverOption}{pdftex}
} { 
  \renewcommand{\driverOption}{dvips}
}

\documentclass[reqno, 11pt, letterpaper, commented, oneside, \driverOption]{amsart}

\newboolean{isCommented}
\usepackage{geometry}
\usepackage{bbm}
\usepackage[final]{graphicx}
\usepackage{upref}
\usepackage{enumerate}
\usepackage{latexsym}
\usepackage{amssymb}
\usepackage[ansinew]{inputenc}
\usepackage[T1]{fontenc}
\usepackage[ruled,vlined]{algorithm2e}
\usepackage{mathtools}
\usepackage{mycommands}

\ifthenelse{\boolean{pdf}}{
  \usepackage[final]{ps4pdf}
} { 
  \usepackage[inactive]{ps4pdf}
}
\PSforPDF{\usepackage{psfrag}}

\newcommand{\hyperrefDriverOption}{}
\ifthenelse{\boolean{pdf}}{
	\renewcommand{\hyperrefDriverOption}{pdftex}
} { 
	\renewcommand{\hyperrefDriverOption}{hypertex}
}
\usepackage[\hyperrefDriverOption,
  colorlinks = false,
  pdfauthor={Torsten Mütze, Reto Spöhel},
  pdftitle={On the path-avoidance vertex-coloring game}]
  {hyperref}

\ifthenelse{\boolean{isCommented}} {
	\newcommand{\TM}[1]{\marginpar{\parbox{4cm}{{\small {\bf TM:} #1}}}}
	\newcommand{\RS}[1]{\marginpar{\parbox{4cm}{{\small {\bf RS:} #1}}}}
} { 
	\newcommand{\TM}[1]{}
	\newcommand{\RS}[1]{}
}

\newtheorem{theorem}{Theorem}
\newtheorem{lemma}[theorem]{Lemma}

\newtheorem{proposition}[theorem]{Proposition}
\newtheorem{corollary}[theorem]{Corollary}

\theoremstyle{definition}

\theoremstyle{remark}
\newtheorem{remark}[theorem]{Remark}

\DeclareMathOperator{\AP}{\textsc{AvoidPaths}}

\long\def\symbolfootnote[#1]#2{\begingroup
\def\thefootnote{\fnsymbol{footnote}}\footnote[#1]{#2}\endgroup}

\ifthenelse{\boolean{isCommented}} {
	\geometry{
	  hmargin={25mm, 50mm},
	  marginparwidth=40mm, 
	  vmargin={25mm, 25mm},
	  headsep=10mm,
	  headheight=5mm,
	  footskip=10mm
	}
} { 
	\geometry{
	  hmargin={24mm, 24mm}, 
	  vmargin={24mm, 24mm}, 
	  headsep=10mm,
	  headheight=5mm,
	  footskip=10mm
	}
}

\begin{document}

\begin{center}

\vspace*{1.5cm}
\LARGE On the path-avoidance vertex-coloring game
\vspace{1cm}

\small

\begingroup
\renewcommand\thefootnote{\fnsymbol{footnote}}%
\begin{tabular}{l@{\hspace{3em}}l}
  \Large Torsten Mütze  & \Large Reto Spöhel\footnotemark[2] \\[2mm]
  Institute of Theoretical Computer Science & Algorithms and Complexity Group \\
  ETH Zürich                       &  Max-Planck-Institut für Informatik \\
  8092 Zürich, Switzerland                    & 66123 Saarbrücken, Germany \vspace{.05mm} \\
  {\small {\tt muetzet@inf.ethz.ch}}       & {\small {\tt rspoehel@mpi-inf.mpg.de}}
\end{tabular}%
\footnotetext[2]{The author was supported by a fellowship of the Swiss National Science Foundation.}%
\endgroup

\vspace{2cm}

\small

\begin{minipage}{0.8\linewidth} \textsc{Abstract.}
For any graph $F$ and any integer $r\geq 2$, the \emph{online vertex-Ramsey density of $F$ and $r$}, denoted $m^*(F,r)$, is a parameter defined via a deterministic two-player Ramsey-type game (Painter vs.\ Builder). This parameter was introduced in a recent paper \cite{mrs11}, where it was shown that the online vertex-Ramsey density determines the threshold of a similar probabilistic one-player game (Painter vs.\ the binomial random graph $G_{n,p}$). For a large class of graphs $F$, including cliques, cycles, complete bipartite graphs, hypercubes, wheels, and stars of arbitrary size, a simple greedy strategy is optimal for Painter and closed formulas for $m^*(F,r)$ are known.

In this work we show that for the case where $F=P_\ell$ is a (long) path, the picture is very different. It is not hard to see that $m^*(P_\ell,r)= 1-1/k^*(P_\ell,r)$ for an appropriately defined integer $k^*(P_\ell,r)$, and that the greedy strategy gives a lower bound of $k^*(P_\ell,r)\geq \ell^r$. We construct and analyze Painter strategies that improve on this greedy lower bound by a factor polynomial in $\ell$, and we show that no superpolynomial improvement is possible.
\end{minipage}

\end{center}

\vspace{5mm}


\section{Introduction}

\subsection{The online vertex-Ramsey density}

Consider the following deterministic two-player game: The two players are called \emph{Builder} and \emph{Painter}, and the board is a vertex-colored graph that grows in each step of the game. Painter wants to avoid creating a monochromatic copy of some fixed graph~$F$, and her opponent Builder wants to force her to create such a monochromatic copy.
The game starts with an empty board, i.e., no vertices are present at the beginning of the game. In each step, Builder presents a new vertex and a number of edges leading from previous vertices to this new vertex. Painter has to color the new vertex immediately and irrevocably with one of $r$ available colors, and she loses as soon as she creates a monochromatic copy of $F$. So far this game would be rather trivial; however, we additionally impose the restriction on Builder that, for some fixed real number~$d$ known to both players, the evolving board~$B$ satisfies $m(B)\leq d$ at all times, where as usual we define
\begin{equation*}
  m(B):=\max_{H\seq B} \frac{e(H)}{v(H)} \enspace,
\end{equation*}
and $e(H)$ and $v(H)$ denote the number of edges and vertices of $H$, respectively.
We will refer to this game as the \emph{$F$-avoidance game with $r$ colors and density restriction~$d$}.

We say that \emph{Builder has a winning strategy} in this game (for a fixed graph~$F$, a fixed number of colors~$r$, and a fixed density restriction $d$) if he can force Painter to create a monochromatic copy of~$F$ within a finite number of steps.
For any graph $F$ and any integer $r\geq 2$ we define the \emph{online vertex-Ramsey density} $m^*(F,r)$ as
\begin{equation} \label{eq:m1*}
  m^*(F,r):=\inf \left\{
    d\in\RR \bigmid
    \leftbox{0.52\displaywidth}{Builder has a winning strategy in the $F$-avoidance game with $r$ colors and density restriction $d$}
  \right\} \enspace.
\end{equation}

The parameter $m^*(F,r)$ was introduced in \cite{mrs11}, where we established a general correspondence between the deterministic two-player game we just introduced, and a similar probabilistic one-player game. (We will explain this correspondence in the next section.) In \cite{mrs11} we also proved the following result.

\begin{theorem}[\cite{mrs11}] \label{thm:main-1}
For any graph $F$ with at least one edge and any integer $r\geq 2$, the online vertex-Ramsey density $m^*(F,r)$ is a computable rational number, and the infimum in \eqref{eq:m1*} is attained as a minimum.
\end{theorem}

To put Theorem~\ref{thm:main-1} into perspective, we mention that none of its three statements (computable, rational, infimum attained as minimum) is known to hold for the offline counterpart of $m^*(F,r)$, i.e., for the \emph{vertex-Ramsey density}
\begin{equation*}
  m^o(F,r):=\inf \left\{
    m(G) \bigmid
    \leftbox{0.38\displaywidth}{every $r$-coloring of the vertices of $G$ contains a monochromatic copy of $F$}
  \right\}
\end{equation*}
introduced in~\cite{MR1248487}. It is also not known whether such statements are true for two analogous parameters related to \emph{edge}-colorings (see~\cite{org-new, MR2152058}). In fact, even the value of $m^o(P_3, 2)$ is unknown --- the authors of \cite{MR1248487} offer 400,000 z{\l}oty (Polish currency in 1993) for its exact determination (here $P_3$ denotes the path on three vertices).

\subsection{Background: a probabilistic one-player game} \label{sec:background}

The main motivation for investigating the deterministic two-player game introduced above comes from the theory of random graphs. More specifically, following work of {\L}uczak, Ruci\'nski, and Voigt~\cite{MR1182457} on vertex-Ramsey properties of random graphs, the following one-player game was studied in \cite{ovg-combinatorica}:
As usual, we denote by $\Gnp$ the random graph on $n$ vertices obtained by including each of the $\binom{n}{2}$ possible edges with probability $p=p(n)$ independently. The vertices of an initially hidden instance of $\Gnp$ are revealed one by one, and at each step of the game only the edges induced by the vertices revealed so far are visible.
As in the deterministic game introduced above, the player Painter immediately and irrevocably assigns one of $r$ available colors to each vertex as soon as it is revealed, with the goal of avoiding monochromatic copies of a fixed graph~$F$. We refer to this game as the \emph{probabilistic $F$-avoidance game with $r$ colors}.

It follows from standard arguments (see~\cite[Lemma 7]{org-lb}) that this game has a threshold $p_0(F,r,n)$ in the following sense: For any function $p(n) = o(p_0)$ there is an online strategy that \aas colors the vertices of $\Gnp$ with $r$ colors without creating a monochromatic copy of $F$, and for any function~$p(n)=\omega(p_0)$ \emph{any} online strategy will \aas fail to do so. (Here \aas stands for `asymptotically almost surely', i.e., with probability tending to 1 as $n$ tends to infinity.)

In a recent paper \cite{mrs11} we extended the results of \cite{ovg-combinatorica} on this probabilistic game and established the following general threshold result.

\begin{theorem}[\cite{mrs11}] \label{thm:main-2}
For any fixed graph $F$ with at least one edge and any fixed integer $r\geq 2$, the threshold of the probabilistic $F$-avoidance game with $r$ colors is
\begin{equation*}
  p_0(F,r,n)=n^{-1/{m^*(F,r)}} \enspace,
\end{equation*}
where $m^*(F,r)$ is defined in \eqref{eq:m1*}.
\end{theorem}

Theorem~\ref{thm:main-2} reduces the problem of determining the threshold of the probabilistic $F$-avoidance game to the purely deterministic combinatorial problem of computing $m^*(F,r)$. Moreover, we can bound the threshold of the probabilistic game by deriving bounds on $m^*(F,r)$, which in turn can be done by designing and analyzing appropriate Painter and Builder strategies for the deterministic $F$-avoidance game.

\subsection{Closed formulas for the online vertex-Ramsey density}

The algorithm presented in \cite{mrs11} to compute $m^*(F,r)$ for general $F$ and $r$ is rather complex and gives no hint as to how the quantity $m^*(F,r)$ behaves for natural graph families. However, for a large class of graphs $F$, a simple closed formula for the parameter $m^*(F,r)$ follows from the results in~\cite{ovg-combinatorica}.
This class includes cliques $K_\ell$, cycles $C_\ell$, complete bipartite graphs $K_{s,t}$, $d$-dimensional hypercubes $Q_d$, wheels $W_\ell$ with $\ell$ spokes, and stars $S_\ell$ with $\ell$ rays. In all those cases, the online vertex-Ramsey density is given by $m^*(F,r)=\frac{e(F)(1-v(F)^{-r})}{v(F)-1}$, i.e., we have
\begin{equation} \label{eq:explicit-m1*}
\begin{split}
  m^*(K_\ell,r) &= \textstyle \frac{\ell(1-\ell^{-r})}{2} \enspace, \\
  m^*(K_{s,t},r) &= \textstyle \frac{st(1-(s+t)^{-r})}{s+t-1} \enspace, \\
  m^*(W_\ell,r) &= \textstyle 2(1-(\ell+1)^{-r}) \enspace,
\end{split}
\qquad
\begin{split}
  m^*(C_\ell,r) &= \textstyle \frac{\ell(1-\ell^{-r})}{\ell-1} \enspace, \\
  m^*(Q_d,r) &= \textstyle \frac{d2^{d-1}(1-2^{-dr})}{2^d-1} \enspace, \\
  m^*(S_\ell,r) &= \textstyle 1-(\ell+1)^{-r} \enspace.
\end{split}
\end{equation}
The reason why the parameter $m^*(F,r)$ has such a simple form in all these cases is that for those graphs $F$ the following simple strategy is optimal for Painter: Assuming the colors are numbered from $1,\dots,r$, the \emph{greedy strategy} in each step uses the highest-numbered color that does not complete a monochromatic copy of $F$ (or color $1$ if no such color exists).

In this work we show that the situation is much more complicated in the innocent-looking case where $F=P_\ell$ is a path on $\ell$ vertices. As it turns out, for this family of graphs the greedy strategy fails quite badly, and the parameter $m^*(P_\ell,r)$ exhibits a much more complex behaviour than one might expect in view of the previous examples.

\subsection{Forests}

We first introduce a more convenient way to express $m^*(F,r)$ for the case where $F$ is an arbitrary forest. Note that a density restriction of the form $d=(k-1)/k$ for some integer $k\geq 2$ is equivalent to requiring that Builder creates no cycles and no components (=trees) with more than $k$ vertices. We call this game the \emph{$F$-avoidance game with $r$ colors and tree size restriction~$k$}.

It is not hard to see that for any forest $F$ and any integer $r\geq 2$, Builder has a winning strategy in the $F$-avoidance game with $r$ colors and tree size restriction $k$ for large enough $k$. We denote by $k^*(F,r)$ the smallest such integer $k$ for which Builder has a winning strategy in this game.

Noting that for any forest $F$ we have
\begin{equation*}
  m^*(F,r)=\frac{k^*(F,r)-1}{k^*(F,r)} \enspace,
\end{equation*}
we obtain the following corollary to Theorem~\ref{thm:main-2}.

\begin{corollary}[\cite{mrs11}] \label{cor:main-2-forests}
For any fixed forest $F$ with at least one edge and any fixed integer $r\geq 2$, the threshold of the probabilistic $F$-avoidance game with $r$ colors is
\begin{equation*}
  p_0(F,r,n)=n^{-1-1/{(k^*(F,r)-1)}} \enspace.
\end{equation*}
\end{corollary}
For the rest of this paper, we restrict our attention to forests and focus on the parameter $k^*(F,r)$. It follows from the results in~\cite{ovg-combinatorica} that for any tree~$F$ and any integer $r\geq 2$ the greedy strategy  guarantees a lower bound of $k^*(F,r)\geq v(F)^r$ (for the sake of completeness we give the argument explicitly in Lemma~\ref{lemma:greedy-lb} below).

\subsection{Our results}

For the rest of this introduction we focus on the case where $F=P_\ell$ and $r=2$ colors are available. Table~\ref{tab:kstar-Pell} shows the exact values of $k^*(P_\ell,2)$ for $\ell\leq 45$. These were determined with the help of a computer, based on the insights of this paper and using some extra tweaks to improve running times (see Section~\ref{sec:implementation} below).
The bottom row shows the difference $k^*(P_\ell,2)-\ell^2$, i.e., by how much optimal Painter strategies can improve on the greedy lower bound $v(P_\ell)^2=\ell^2$.

\begin{table}
\scriptsize
\setlength{\tabcolsep}{0.8mm}
\begin{center}
\begin{tabular}{|l||r|r|r|r|r|r|r|r|r|r|r|r|r|r|r|r|r|r|r|}\hline
$\ell$ & $2,\ldots,27$ & 28 & 29 & 30 & 31 & 32 & 33 & 34 & 35 & 36 & 37 & 38 & 39 & 40 & 41 & 42 & 43 & 44 & 45 \\ \hline\hline
$k^*(P_\ell,2)$ & $2^2,\ldots,27^2$ & 791 & 841 & 902 & 961 & 1040 & 1089 & 1156 & 1225 & 1323 & 1376 & 1449 & 1521 & 1641 & 1699 & 1796 & 1856 & 1991 & 2057 \\
$k^*(P_\ell,2)-\ell^2$ & 0 & 7 & 0 & 2 & 0 & 16 & 0 & 0 & 0 & 27 & 7 & 5 & 0 & 41 & 18 & 32 & 7 & 55 & 32 \\ \hline
\end{tabular}
\end{center}
\caption{Exact values of $k^*(P_\ell,2)$ for $\ell\leq 45$.} \label{tab:kstar-Pell}
\end{table}

In stark contrast to the formulas in~\eqref{eq:explicit-m1*}, the values in Table~\ref{tab:kstar-Pell} (and the corresponding optimal Painter strategies) exhibit a rather irregular behaviour and seem to follow no discernible pattern. In particular, the greedy strategy turns out to be optimal for $\ell\in\{2,\ldots,27\}\cup\{29,31,33,34,35,39\}$, but not for the other values of $\ell\leq 45$. (In fact, for all $\ell\geq 46$ we have $k^*(P_\ell,2)>\ell^2$, so the listed values are the only ones for which the greedy strategy is optimal.)

These numerical findings raise the question whether and by how much optimal Painter strategies can improve on the greedy lower bound asymptotically as $\ell\to\infty$. Our main result shows that there exist Painter strategies that improve on the greedy lower bound by a factor polynomial in $\ell$, and that no superpolynomial improvement is possible.

\begin{theorem}[Main result] \label{thm:kstar-Pell-2} 
We have
\begin{equation*}
  \Theta(\ell^{2.01})\leq k^*(P_\ell,2) \leq \Theta(\ell^{2.59}) \enspace.
\end{equation*}
\end{theorem}

We prove the bounds in Theorem~\ref{thm:kstar-Pell-2} by analyzing a more general asymmetric version of the path-avoidance game, where Painter's goal is to avoid a path on $\ell$ vertices in color $1$, and a path on $c$ vertices in color $2$. We denote by $k^*(P_\ell,P_c)$ the smallest integer $k$ for which Builder has a winning strategy in this asymmetric $(P_\ell,P_c)$-avoidance game with tree size restriction $k$. 

In the following we present our results for this asymmetric game. The next theorem shows in particular that for any fixed value of $c$, the parameter $k^*(P_\ell,P_c)$ grows linearly with $\ell$.

\begin{theorem} \label{thm:kstar-Pell-Pc}
For any $c\geq 1$ there is a constant $\delta(c)$ such that for any $\ell\geq 1$ we have
\begin{equation*}
  k^*(P_\ell,P_c)=(\delta(c)-o(1))\cdot\ell \enspace,
\end{equation*}
where $o(1)$ stands for a non-negative function of $c$ and $\ell$ that tends to 0 for $c$ fixed and $\ell\rightarrow\infty$.
\end{theorem}

Note that Theorem~\ref{thm:kstar-Pell-Pc} does \emph{not} imply that $k^*(P_\ell,2)=(\delta(\ell)-o(1))\cdot\ell$ as $\ell\rightarrow\infty$.

Similarly to the symmetric game, the greedy strategy guarantees a lower bound of $k^*(P_\ell,P_c)\geq c\cdot\ell$, and it is not hard to see that this is an exact equality for $c\in\{1,2,3\}$ (i.e., the greedy strategy is optimal, see Lemmas~\ref{lemma:greedy-lb} and~\ref{lemma:kstar-F1-F2-ub} below). Thus the constant $\delta(c)$ from Theorem~\ref{thm:kstar-Pell-Pc} satisfies $\delta(c)=c$ for $c\in\{1,2,3\}$. The next theorem states the exact value of $\delta(c)$ for $c\in\{4,5,6\}$. Perhaps surprisingly, these values turn out to be irrational.

\begin{theorem} \label{thm:delta-1to6}
For the constant $\delta(c)$ from Theorem~\ref{thm:kstar-Pell-Pc} we have
\begin{align*}
  \delta(4)&= \textstyle \frac{1}{2}(\sqrt{13}+5)= 4.302\ldots \enspace, \\
  \delta(5)&= \textstyle \frac{1}{2}(\sqrt{24}+6)= 5.449\ldots \enspace, \\
  \delta(6)&= \textstyle \frac{1}{2}(\sqrt{37}+7)= 6.541\ldots \enspace.
\end{align*}
\end{theorem}

Our last result bounds the asymptotic growth of the constant $\delta(c)$ from Theorem~\ref{thm:kstar-Pell-Pc}.

\begin{theorem} \label{thm:delta-c-bounds}
As a function of $c$, the constant $\delta(c)$ from Theorem~\ref{thm:kstar-Pell-Pc} satisfies
\begin{equation*}
  \Theta(c^{1.05}) \leq \delta(c) \leq \Theta(c^{1.59}) \enspace.
\end{equation*}
\end{theorem}

Note that the upper bound in Theorem~\ref{thm:kstar-Pell-2} follows immediately by combining Theorem~\ref{thm:kstar-Pell-Pc}  with the upper bound on $\delta(c)$ stated in Theorem~\ref{thm:delta-c-bounds}, using the non-negativity of the $o(1)$ term in Theorem~\ref{thm:kstar-Pell-Pc}.

\subsection{About the proofs} \label{sec:main-ideas}

We conclude this introduction by highlighting some of the key features in our proofs in an informal way.

As it turns out, the family of all `reasonable' Painter strategies in the $P_\ell$-avoidance game with $r=2$ colors is in one-to-one correspondence with monotone walks from $(1,1)$ to $(\ell,\ell)$ in the integer lattice $\mathbb{Z}^2$. Such a walk is interpreted as follows: If the walk goes from $(x,y)$ to $(x+1,y)$, Painter will use color~$1$ when faced with the decision of either creating a $P_x$ in color~$1$ or a $P_y$ in color~$2$. Conversely, a step from $(x,y)$ to $(x,y+1)$ indicates that Painter uses color~$2$ in the same situation. (The greedy strategy corresponds to the walk that goes from $(1,1)$ first to $(1,\ell)$ and then to $(\ell,\ell)$.)
Note that there are $\binom{2(\ell-1)}{\ell-1}=4^{(1+o(1))\ell}$ such walks, and thus the same number of `candidate strategies' for Painter.

For any fixed such walk, we can compute the smallest tree size restriction that allows Builder to enforce a monochromatic copy of $P_\ell$ against this particular Painter strategy by a recursive computation along the walk. This recursion involves only integers and no complicated tree structures. We can then compute the parameter $k^*(P_\ell,2)$ by performing this recursive computation for all (exponentially many) walks of the described form, and taking the maximum. (This entire procedure can be seen as a highly specialized form of the general algorithm for computing $m^*(F,r)$ given in~\cite{mrs11}.)
With these insights in hand, understanding the vertex-coloring path-avoidance game reduces to the algebraic problem of understanding this recursion along lattice walks.

The lattice walks (i.e.\ Painter strategies) yielding the lower bounds in Theorem~\ref{thm:kstar-Pell-2} and Theorem~\ref{thm:delta-c-bounds} have an interesting self-similar structure: essentially, they are obtained by nesting a large number of copies of a nearly-optimal walk for the asymmetric $(P_\ell,P_4)$-avoidance game at different scales into each other (see Figure~\ref{fig:bootstrapping} below).

\subsection{Organization of this paper}

In Section~\ref{sec:observations} we collect a few general observations about the $F$-avoidance game for the case where $F$ is a forest. In Section~\ref{sec:recursion} we turn to the case of paths and present the recursion that allows us to compute the parameter $k^*(P_\ell,2)$ (or more generally, the parameter $k^*(P_\ell,P_c)$).
This recursion is analyzed in Section~\ref{sec:analyze-recursion} to derive Theorems~\ref{thm:kstar-Pell-2}--\ref{thm:delta-c-bounds}.

\section{Basic observations} \label{sec:observations}

For our proofs we will consider the general asymmetric $(F_1,\ldots,F_r)$-avoidance game, where Painter's goal is to avoid a (possibly different) forest $F_s$ in each color $s\in[r]$. We denote by $k^*(F_1,\ldots,F_r)$ the smallest integer $k$ for which Builder has a winning strategy in this asymmetric $(F_1,\ldots,F_r)$-avoidance game with tree size restriction $k$.

In this section we prove straightforward lower and upper bounds for this parameter (Lemma~\ref{lemma:greedy-lb} and Lemma~\ref{lemma:kstar-F1-F2-ub} below). These lemmas show that the constant $\delta(c)$ from Theorem~\ref{thm:kstar-Pell-Pc} satisfies $\delta(c)=c$ for $c\in\{1,2,3\}$, and their proofs also serve as a warm-up for the reader to get familiar with the type of reasoning that is used throughout the paper.

The definition of the greedy strategy extends straightforwardly to the general asymmetric $(F_1,\ldots,F_r)$-avoidance game: This strategy in each step uses the highest-numbered color $s\in[r]$ that does not complete a monochromatic copy of $F_s$ (or color $1$ if no such color exists).

\begin{lemma}[Greedy lower bound] \label{lemma:greedy-lb}
For any trees $F_1,\ldots,F_r$, we have $k^*(F_1,\ldots,F_r)\geq v(F_1)\cdots v(F_r)$.
\end{lemma}

\begin{proof}
We show that the greedy strategy is a winning strategy for Painter in the game with tree size restriction $v(F_1)\cdots v(F_r)-1$. Suppose for the sake of contradiction that Painter loses this game when playing the greedy strategy. Then, by the definition of the strategy, the board contains a copy of $F_1$ in color~1. Moreover, each vertex $v$ in color~1 in this copy is adjacent to a set of trees in color~2 which together with $v$ form a copy of $F_2$, so the board contains a tree on $v(F_1)\cdot v(F_2)$ vertices in the colors~1 or~2. Continuing this argument inductively, we obtain that for all $k=2,\ldots,r$ each vertex $v$ in one of the colors $\{1,\ldots,k-1\}$ is adjacent to a set of trees in color~$k$ which together with $v$ form a copy of $F_k$, and that consequently  the board contains a tree on $v(F_1)\cdots v(F_k)$ vertices in colors from $\{1,\ldots,k\}$. For $k=r$ this yields the desired contradiction.
\end{proof}

Observe that if Builder confronts Painter several times with the decision on how to color a new vertex that connects in the same way to copies of the same $r$-colored trees, then by the pigeonhole principle, Painter's decision will be the same in at least a $(1/r)$-fraction of the cases. As a consequence, we can assume \wolog that Painter plays \emph{consistently} in the sense that her strategy is determined by a function that maps unordered tuples of $r$-colored rooted trees to the set of available colors $\{1,\ldots,r\}$ (with the obvious interpretation that Painter uses the corresponding color whenever a new vertex connects exactly to the roots of copies of the trees in such a tuple). This assumption is very useful when proving upper bounds for $k^*(F_1,\ldots,F_r)$ by describing explicit strategies for Builder, as it implies that if Builder has enforced a copy of some tree on the board, then he can enforce as many additional copies of this tree as he needs. We thus avoid the hassle of making the repetitive pigeonholing steps for Builder explicit.

For the following lemma recall that we denote by $S_\ell$ the star with $\ell$ rays.

\begin{lemma}[Tree versus star upper bound] \label{lemma:kstar-F1-F2-ub}
For any tree $F$ and any $\ell\geq 1$ we have $k^*(F,S_\ell)\leq v(F)\cdot v(S_\ell)=v(F)\cdot(\ell+1)$.
\end{lemma}

Note that this bound matches the greedy lower bound given by the previous lemma. It follows in particular that $k^*(P_\ell, P_c)=c\cdot \ell$ for any $\ell\geq 1$ and $c\in\{1,2,3\}$.

For the proof of Lemma~\ref{lemma:kstar-F1-F2-ub} we use the following auxiliary lemma (for a proof see e.g.~\cite{Vygen201167}).

\begin{lemma}[Tree splitting] \label{lemma:tree-splitting}
For any tree $F$ and any integer $s\geq 1$ there is a subset $S\seq V(F)$ with $|S|\leq \lfloor\frac{v(F)}{s}\rfloor$ such that when removing the vertices of $S$ from $F$ all remaining components (=trees) have at most $s-1$ vertices.
\end{lemma}

\begin{proof}[Proof of Lemma~\ref{lemma:kstar-F1-F2-ub}]
We describe a winning strategy for Builder in the $(F,S_\ell)$-avoidance game with tree size restriction $v(F)\cdot v(S_\ell)$. We may and will assume \wolog that Painter plays consistently as defined above, implying that if Builder has enforced a copy of some tree on the board, then he can enforce as many additional copies of this tree as he needs.

Builder's strategy works in two phases. The first phase lasts as long as Painter continues using color~1, and ends when she uses color~2 for the first time. In the first phase, for $n=1,2,\ldots$ Builder enforces copies of \emph{all} trees with exactly $n$ vertices in color~1 (first all trees with one vertex, then all trees with two vertices and so on; all those copies are isolated, i.e., they are not connected to other parts of the board). Let $s$ denote the value of $n$ when Painter uses color~2 for the first time. At this point Builder has enforced, for each $n\leq s-1$, a copy of every tree on $n$ vertices in color~1, and a single vertex in color~2 that is contained in a tree $T$ with $v(T)=s$ vertices.

For the second phase, apply Lemma~\ref{lemma:tree-splitting} and fix a subset $S\seq V(F)$ with $|S|\leq \lfloor\frac{v(F)}{s}\rfloor$ such that when removing the vertices of $S$ from $F$ all remaining components (=trees) have at most $s-1$ vertices. In this phase Builder uses copies of the components in $F\setminus S$ in color~1 from the first phase and connects them with $|S|$ many new vertices in such a way that assigning color~1 to all of these new vertices would create a copy of $F$ in color~1.
At the same time, Builder also connects each of these new vertices to the vertex in color~2 of $\ell$ separate copies of $T$, such that assigning color 2 to any of the new vertices would create a copy of $S_\ell$ in color~2. (In total Builder uses $\ell\cdot|S|$ many copies of $T$.) Hence the game ends either with a copy of $F$ in color~1 or a copy of $S_\ell$ in color~2, and the number of vertices of the largest component (=tree) Builder constructs during the game is
\begin{equation*}
  v(F)+\ell\cdot|S|\cdot v(T) \leq v(F)+\ell\cdot \Big\lfloor\frac{v(F)}{s}\Big\rfloor\cdot s \leq v(F)\cdot (\ell+1)=v(F)\cdot v(S_\ell) \enspace,
\end{equation*}
proving the lemma.
\end{proof}

\section{A general recursion} \label{sec:recursion}

In this section we derive a general recursion that allows us to compute the parameter $k^*(P_{\ell_1},\ldots,P_{\ell_r})$  for arbitrary values $\ell_1, \ldots, \ell_r\geq 1$ (see Proposition~\ref{prop:kstar-P1-Pr} below). This turns the problem of analyzing the $(P_{\ell_1},\ldots,P_{\ell_r})$-avoidance game into the algebraic problem of analyzing this recursion.
As innocent as this recursion may look, it generates surprisingly complex patterns, which surface only for relatively large values of $\ell_1,\ldots,\ell_r$ (recall Table~\ref{tab:kstar-Pell} for the special case $r=2$, $\ell_1=\ell_2=\ell$). Understanding the asymptotic features of this recursion will be the key to proving Theorems~\ref{thm:kstar-Pell-2}--\ref{thm:delta-c-bounds}.

Throughout this section we include the case with more than two colors. There is little overhead for doing so, and it is notationally convenient to distinguish indices $s\in[r]$ referring to colors from certain indices $1$ and $2$ that appear otherwise. 

\subsection{A recursion along lattice walks}
\label{sec:defining-the-recursion}

Let $\alpha=(\alpha_i)_{i\geq 1}$ be an infinite sequence with entries from the set $[r]$. For any $i\geq 0$ and any $s\in[r]$ we define
\begin{subequations} \label{eq:recursion}
\begin{equation} \label{eq:nu-i-s}
  \nu_{i,s}:=1+|\{1\leq j\leq i \mid \alpha_j=s \}| \enspace.
\end{equation}
It is convenient to think of $\alpha$ as an increasing axis-parallel walk in the $r$-dimensional integer lattice $\mathbb{Z}^r$ with starting point $(1,1,\ldots,1)$, where in the $i$-th step of the walk the current position changes by $+1$ in the coordinate direction $\alpha_i$. Note that $\nu_i=(\nu_{i,1},\ldots,\nu_{i,r})$ as defined in \eqref{eq:nu-i-s} denotes the position of the walk after the first $i$ steps.

The recursion defined below is parametrized by such a sequence $\alpha=(\alpha_i)_{i\geq 1}$, $\alpha_i\in[r]$, where this sequence can be interpreted as a strategy for Painter in some $(P_{\ell_1},\ldots,P_{\ell_r})$-avoidance game as follows: For any point $\nu_i$, $i\geq 0$, on the walk corresponding to $\alpha$, whenever the longest path that would be created by assigning color $s$ to a new vertex on the board is $\nu_{i,s}$ for each color $s\in[r]$, Painter chooses color~$\sigma:=\alpha_{i+1}$ (i.e., she prefers completing a path on $\nu_{i,\sigma}$ vertices in color $\sigma$ over the other alternatives). To obtain a fully defined Painter strategy we will extend this criterion using certain natural monotonicity conditions: If e.g.\ Painter prefers a $P_5$ in color~1 over a $P_7$ in color~2, she will also prefer a $P_5$ in color~1 over a $P_8$ in color~2. The precise strategy definition is given below in the proof of Proposition~\ref{prop:kstar-P1-Pr}. The recursion defined in the following evaluates the performance of the strategy corresponding to the given sequence $\alpha$.

For a given sequence $\alpha=(\alpha_i)_{i\geq 1}$, $\alpha_i\in[r]$, the recursion computes an infinite sequence of integers $(k_i)_{i\geq 0}$. As auxiliary variables it maintains sequences of integers $x_1,x_2,\ldots,x_r$, where for each $s\in[r]$ we write $x_s=(x_{s,0},x_{s,1},\ldots)$. (To simplify notation we suppress the dependence of the values $k_i$, of the sequences $x_s$ and of the values $\nu_{i,s}$ defined in \eqref{eq:nu-i-s} from the parameter~$\alpha$.)

For each $i\geq 0$, first $k_i$ is computed, and then this value is appended to exactly one of the sequences $x_1,\ldots,x_r$, namely to the sequence specified by $\alpha_{i+1}$.
Specifically, for each $s\in[r]$ we define
\begin{equation} \label{eq:x-s-0}
  x_{s,0}:=0 \enspace,
\end{equation}
and for any $i\geq 0$ we define
\begin{equation} \label{eq:k-i}
  k_i:=1+\sum_{s\in[r]} \min_{\substack{j_1,j_2\geq 0: \\ j_1+j_2=\nu_{i,s}-1}}(x_{s,j_1}+x_{s,j_2})
\end{equation}
and
\begin{equation} \label{eq:x-sigma}
  x_{s,\nu_{i,s}}:=k_i \qquad \text{if} \;\; \alpha_{i+1}=s \enspace,
\end{equation}
\end{subequations}
where the values $\nu_{i,s}$ are defined in \eqref{eq:nu-i-s} for the given sequence $\alpha$. (One can check that after step $i$ of the recursion exactly the values $k_0, \ldots, k_i$ and, for each $s\in[r]$, the values $x_{s,0},\ldots, x_{s,\nu_{i+1,s}-1}$ have been computed.)
An example illustrating these definitions is given in Figure~\ref{fig:recursion}.

Note that we can think of the sequence $(k_i)_{i\geq 0}$ as being computed along the walk corresponding to $\alpha$, and for each $s\in[r]$ the entries of the sequence $x_s$ are obtained by selecting those values $(k_i)_{i\geq 0}$ where the walk takes a step in direction~$s$ (see Figure~\ref{fig:recursion}). As we shall see, for any $s\in[r]$ and any $j\geq 0$ the number $x_{s,j}$ equals the number of vertices in the smallest component (=tree) containing a path on $j$ vertices in color $s$ if Painter plays according to the strategy corresponding to the sequence $\alpha$ (see Lemma~\ref{lemma:strategy-invariant} below).

The following lemma is an immediate consequence of the definitions in \eqref{eq:recursion}.

\begin{lemma}[Monotonicity along the recursion] \label{lemma:monotonicity}
For any $\alpha=(\alpha_i)_{i\geq 1}$, $\alpha_i\in[r]$, the sequence $(k_i)_{i\geq 0}$ and in particular each of the sequences $x_1,\ldots,x_r$ defined in \eqref{eq:recursion} is strictly increasing.
\end{lemma}

In the following we are only interested in evaluating the above recursion for a finite number of steps. More specifically, for integers $\ell_1,\ldots,\ell_r\geq 1$ we denote by $W(\ell_1,\ldots,\ell_r)$ the set of finite sequences of length
\begin{subequations} \label{eq:d-k-alpha}
\begin{equation} \label{eq:d}
  d=d(\ell_1,\ldots,\ell_r):=\sum_{s\in[r]}(\ell_s-1)
\end{equation}
with the property that for each $s\in[r]$, exactly $\ell_s-1$ entries are equal to $s$ (i.e., the walk corresponding to such a sequence ends at $(\ell_1,\ldots,\ell_r)$, see Figure~\ref{fig:recursion}).
For any such $\alpha\in W(\ell_1,\ldots,\ell_r)$, we may evaluate the recursion \eqref{eq:recursion} for the first $d+1$ steps (i.e., for $i=0,\ldots, d$), and define
\begin{equation} \label{eq:k-alpha}
  k(\alpha):=k_d \enspace.
\end{equation}
\end{subequations}
(In the last step $i=d$, \eqref{eq:x-sigma} should be ignored.)

\begin{figure}
\centering
\PSforPDF{
 \psfrag{col1}{Color 1}
 \psfrag{col2}{Color 2}
 \psfrag{1}{1}
 \psfrag{2}{2}
 \psfrag{3}{3}
 \psfrag{4}{4}
 \psfrag{5}{5}
 \psfrag{6}{6}
 \psfrag{n0}{$\nu_0$}
 \psfrag{n1}{$\nu_1$}
 \psfrag{n2}{$\nu_2$}
 \psfrag{n3}{$\nu_3$}
 \psfrag{n4}{$\nu_4$}
 \psfrag{n5}{$\nu_5$}
 \psfrag{n6}{$\nu_6$}
 \psfrag{n7}{$\nu_7$}
 \psfrag{k0}{\small $k_0=1$}
 \psfrag{k1}{\small $k_1=2$}
 \psfrag{k2}{\small $k_2=3$}
 \psfrag{k3}{\small $k_3=6$}
 \psfrag{k4}{\small $k_4=9$}
 \psfrag{k5}{\small $k_5=10$}
 \psfrag{k6}{\small $k_6=11$}
 \psfrag{k71}{\small $k_7=1+\min(0+11,1+10,2+9)$}
 \psfrag{k72}{\small ${}+\min(0+6,3+3)=18$}
 \psfrag{kd}{\small $k_d$}
 \psfrag{alpha}{$\alpha=(\alpha_1,\alpha_2,\ldots)=(1,1,2,2,1,1,1,2,\ldots)$}
 \psfrag{x1}{$x_1=(0,k_0,\quad\;\;\,k_1,\quad\;\;\,k_4,\quad\;\;\,k_5,\quad\;\;\,k_6,\quad\;\;\,\ldots\;\;)$}
 \psfrag{x2}[l][c][1][90]{$x_2=(0,k_2,\quad\;\;\,k_3,\quad\;\;\,k_7,\quad\;\;\,\ldots\;\;)$}
 \psfrag{l1}{$\ell_1$}
 \psfrag{l2}{$\ell_2$}
 \includegraphics{graphics/path-recursion.eps}
}
\caption{Illustration of the definitions in \eqref{eq:recursion} and \eqref{eq:d-k-alpha} for the case $r=2$.} \label{fig:recursion}
\end{figure}

The following proposition is the main result of this section and characterizes the parameter $k^*(P_{\ell_1},\ldots,P_{\ell_r})$ from the $(P_{\ell_1},\ldots,P_{\ell_r})$-avoidance game in terms of the recursion defined above.

\begin{proposition}[General recursion] \label{prop:kstar-P1-Pr}
For any integers $\ell_1,\ldots,\ell_r\geq 1$, we have
\begin{equation} \label{eq:kstar-P1-Pr}
  k^*(P_{\ell_1},\ldots,P_{\ell_r})=\max_{\alpha\in W(\ell_1,\ldots,\ell_r)} k(\alpha) \enspace,
\end{equation}
where $k(\alpha)$ is defined in \eqref{eq:recursion} and \eqref{eq:d-k-alpha}.
\end{proposition}

\subsection{Proof of Proposition~\texorpdfstring{\ref{prop:kstar-P1-Pr}}{12}}

We begin by proving that the right hand side of \eqref{eq:kstar-P1-Pr} is an upper bound on $k^*(P_{\ell_1},\ldots,P_{\ell_r})$. We do so by describing a Builder strategy that closely resembles the structure of the recursion~\eqref{eq:recursion}.

\begin{proof}[Proof of Proposition~\ref{prop:kstar-P1-Pr} (upper bound)]
We describe a winning strategy for Builder in the $(P_{\ell_1},\ldots,P_{\ell_r})$-avoidance game with tree size restriction
\begin{equation} \label{eq:restriction}
  k := \max_{\alpha\in W(\ell_1,\ldots,\ell_r)} k(\alpha) \enspace.
\end{equation}
We may and will assume \wolog that Painter plays consistently in the sense of Section~\ref{sec:observations}, implying that if Builder has enforced a copy of some tree on the board, then he can enforce as many additional copies of this tree as he needs. Moreover, we will ignore such repeated steps when counting the number of steps it takes until Builder has enforced a copy of some tree on the board.
Intuitively, Builder's strategy follows the recursion defined in \eqref{eq:recursion} for a sequence $\alpha=(\alpha_i)_{i\geq 1}$, $\alpha_i\in[r]$, that is extracted step by step from Painter's coloring decisions during the game.

Specifically, Builder maintains in each color $s\in[r]$ a list $T_s=(T_{s,0},T_{s,1},\ldots,T_{s,\nu_s-1})$, where $T_{s,0}$ is the null graph ($v(T_{s,0})=0$) and $T_{s,j}$, $1\leq j\leq \nu_s-1$, is a tree containing a monochromatic $P_j$ in color $s$ for which Builder has already enforced a copy on the board. Initially, we have $T_s=(T_{s,0})$ for all $s\in[r]$. In each step, Builder does the following: Given the lists $T_s=(T_{s,0},\ldots,T_{s,\nu_s-1})$, $s\in[r]$, he adds a new vertex $v$ to the board and, for each color $s\in[r]$, connects it to copies of two trees from the list $T_s$ for which the sum $v(T_{s,j_1})+v(T_{s,j_2})$, $j_1+j_2=\nu_s-1$, is minimized, in such a way that if Painter assigns color~$s$ to $v$, a path on $j_1+j_2+1=\nu_s$ many vertices in color~$s$ is created (if one of the contributing graphs is the null graph, then no corresponding edge is added). Let $\sigma\in[r]$ denote the color Painter assigns to $v$, thus creating a tree that contains a copy of $P_{\nu_\sigma}$ in color $\sigma$. If $\nu_\sigma<\ell_\sigma$, then Builder adds this tree to the end of the list $T_\sigma$, which therefore grows by one element. Otherwise the game ends with a monochromatic $P_{\ell_\sigma}$ in color $\sigma$. Let $d'+1$ denote the number of steps until the game ends (we consider these steps indexed from $0$ to $d'$), and $\alpha'\in[r]^{\{1,\ldots,d'\}}$ the sequence of all coloring decisions of Painter except the last one during Builder's strategy. (Thus Painter's decision in step $i$, $0\leq i\leq d'-1$, is given by $\alpha'_{i+1}$, in line with~\eqref{eq:x-sigma}.) As each time Painter uses some color $s\in[r]$ the length of the list $T_s$ grows by exactly one, the sequence $\alpha'$ has at most $\ell_s-1$ entries equal to~$s$.

It follows easily by induction that this Builder strategy satisfies the following property:
For each $0\leq i\leq d'$ the lists $T_s=(T_{s,0},\ldots,T_{s,\nu_{i,s}-1})$, $s\in[r]$, satisfy
\begin{equation*}  (v(T_{s,0}),v(T_{s,1}),\ldots,v(T_{s,\nu_{i,s}-1}))=(x_{s,0},\ldots,x_{s,\nu_{i,s}-1})\enspace,
\end{equation*}
and the tree constructed in step $i$ has $k_i$ many vertices, where $\nu_{i,s}$, $k_i$ and the sequences $x_1,\ldots,x_s$ are defined in \eqref{eq:recursion} for the given $\alpha'$.

From this property it follows with Lemma~\ref{lemma:monotonicity} that the largest tree Builder constructs is the one in the last step of the game, and that it has $k_{d'}$ many vertices. Letting $\alpha$ denote any sequence from the set $W(\ell_1,\ldots,\ell_r)$ with prefix $\alpha'$, and $k_d$ (with $d$ as in \eqref{eq:d}) the value defined in \eqref{eq:recursion} for this $\alpha$, we obtain with Lemma~\ref{lemma:monotonicity} that
\begin{equation*}
  k_{d'} \leq k_d \eqBy{eq:k-alpha} k(\alpha) \leBy{eq:restriction} k \enspace,
\end{equation*}
showing that Builder adhered to the given tree size restriction.
\end{proof}

For proving the lower bound in Proposition~\ref{prop:kstar-P1-Pr} we will need the following observation. (If the reader is deterred by the technical-looking statement, we recommend looking at the very elementary proof first.)

\begin{lemma}[Choosing a color] \label{lemma:choose-color}
Let $\ell_1,\ldots,\ell_r\geq 1$ be integers and $\alpha\in W(\ell_1,\ldots,\ell_r)$. Then for any integers $\lambda_1,\ldots,\lambda_r$ with $1\leq \lambda_s\leq \ell_s$, $s\in[r]$, and $\lambda_s<\ell_s$ for at least one $s\in[r]$, the following holds: There is a unique integer $0\leq i\leq d-1$ such that for $\sigma:=\alpha_{i+1}$ we have
\begin{subequations} \label{eq:nu-s-d-s}
\begin{align}
  \nu_{i,\sigma} &= \lambda_\sigma \enspace, \label{eq:nu-sigma-d-sigma} \\
  \nu_{i,s}      &\leq \lambda_s \enspace, \quad s\in[r]\setminus\{\sigma\} \enspace,
\end{align}
\end{subequations}
where $\nu_{i,s}$, $s\in[r]$, is defined in \eqref{eq:nu-i-s} for the given $\alpha$ and $d=d(\ell_1,\ldots,\ell_r)$ is defined in~\eqref{eq:d}.
Moreover, we then have $\lambda_\sigma<\ell_\sigma$.
\end{lemma}

\begin{proof}
Geometrically, the box $B:=[1,\lambda_1]\times\cdots\times[1,\lambda_r]$ is contained in the larger box $[1,\ell_1]\times\cdots\times[1,\ell_r]$. As the walk corresponding to the sequence $\alpha$ starts at $(1,\ldots,1)$ and ends at $(\ell_1,\ldots,\ell_r)$, there is a unique first step where it leaves the box $B$. It is easy to see that the starting point $\nu_i$ of this step (which lies on the boundary of $B$) is the unique integer $i$ that satisfies the conditions of the lemma.
\end{proof}

Consider now the following Painter strategy for the $(P_{\ell_1},\ldots,P_{\ell_r})$-avoidance game, which is defined for an arbitrary fixed $\alpha\in W(\ell_1,\ldots,\ell_r)$, and which we denote by $\AP_\alpha(P_{\ell_1},\ldots,P_{\ell_r})$.
For each new vertex $v$, Painter determines for each color $s\in[r]$ the number of vertices $\lambda_s'$ of the longest monochromatic path in color $s$ that would be completed if that color were assigned to $v$, and defines $\lambda_s:=\min(\lambda_s',\ell_s)$.
If $(\lambda_1,\ldots,\lambda_r)=(\ell_1,\ldots,\ell_r)$, then she assigns an arbitrary color to $v$ (and the game ends). Otherwise one of the values $\lambda_s$ is strictly smaller than $\ell_s$. Painter then chooses an $0\leq i\leq d-1$ such that for $\sigma:=\alpha_{i+1}$ the relations \eqref{eq:nu-s-d-s} hold (such a choice is possible by Lemma~\ref{lemma:choose-color}), and assigns color~$\sigma$ to $v$. (As we have $\lambda_\sigma<\ell_\sigma$ in this case, this does not create a monochromatic $P_{\ell_\sigma}$ in color $\sigma$. Moreover, using color~$\sigma$ does not increase the length of any monochromatic path in a color different from $\sigma$, implying that the game does not end in this step.)

For the rest of this paper we usually refer to a sequence $\alpha\in W(\ell_1,\ldots,\ell_r)$ as a \emph{strategy sequence}, having the above interpretation in mind. Note that the greedy strategy analyzed in Lemma~\ref{lemma:greedy-lb} is exactly $\AP_\alpha(P_{\ell_1},\ldots,P_{\ell_r})$ for the strategy sequence $\alpha=(r)^{\ell_r-1}\circ (r-1)^{\ell_{r-1}-1}\circ\cdots\circ(1)^{\ell_1-1}$. Here and throughout we use $\circ$ to denote concatenation of sequences, and integer exponents to indicate repetitions.

The next lemma states the strategy invariant that we already briefly mentioned when we introduced the recursion~\eqref{eq:recursion}. 

\begin{lemma}[Strategy invariant] \label{lemma:strategy-invariant}
Playing according to the strategy $\AP_\alpha(P_{\ell_1},\ldots,P_{\ell_r})$ ensures that the following invariant holds throughout (except possibly in the last step when the game ends): For each $s\in[r]$ and each $0\leq t\leq \ell_s-1$, each monochromatic $P_t$ in color $s$ on the board is contained in a component (=tree) with at least $x_{s,t}$ vertices, where $x_{s,t}$ is defined in \eqref{eq:recursion} for the given $\alpha$.
\end{lemma}

As we shall see, the above invariant is also maintained in the last step when the game ends, but for technical reasons we do not prove this here.

\begin{proof}
To show that this invariant holds, we argue by induction over the number of steps of the game: Initially, no graph is present on the board, and the statement is trivially true (with $t=0$ and $x_{s,0}=0$). For the induction step consider a fixed step where the game does not end, and let $\lambda_s$, $s\in[r]$, be as defined in Painter's strategy.
Furthermore, let $i$ denote the index guaranteed by Lemma~\ref{lemma:choose-color} for these values $\lambda_s$, and let $\sigma=\alpha_{i+1}$ denote the color Painter assigns to the new vertex $v$ in this step.
Clearly, the invariant is maintained for all colors $s\in[r]\setminus\{\sigma\}$, and it remains to show that it holds for $\sigma$. By Lemma~\ref{lemma:monotonicity} we have
\begin{equation*} \label{eq:x-monotonicity}
  x_{\sigma,0}< x_{\sigma,1}< \cdots< x_{\sigma,\ell_\sigma-1} \enspace,
\end{equation*}
implying that it suffices to consider a \emph{longest} monochromatic path in color $\sigma$ that is completed by Painter's decision to assign color $\sigma$ to $v$. Let $Q_\sigma$ denote such a path, and set $t:=v(Q_\sigma)$ (as the game does not end in the current step we have $t\leq \ell_\sigma-1$).
By definition of Painter's strategy, we have
\begin{equation} \label{eq:d-sigma-lambda}
  \lambda_\sigma=t \enspace,
\end{equation}
and for each $s\in[r]\setminus\{\sigma\}$, assigning color~$s$ to $v$ would have completed some (not necessarily maximal) path $Q_s$ in color $s$ on $\lambda_s$ vertices.
Note that the paths $Q_1,\ldots,Q_r$ only share the vertex $v$, and that $v$ divides each of these paths into two paths $Q_{s,1}$ and $Q_{s,2}$ which for $j_{s,1}:=v(Q_{s,1})$ and $j_{s,2}:=v(Q_{s,2})$ satisfy
\begin{equation} \label{eq:js12-sum}
  j_{s,1}+j_{s,2} = \lambda_s-1 \geBy{eq:nu-s-d-s} \nu_{i,s}-1 \enspace.
\end{equation}
Furthermore, observe that the $2r$ paths $Q_{s,1}$ and $Q_{s,2}$, $s\in[r]$, were contained in $2r$ distinct components (=trees) $T_{s,1}$ and $T_{s,2}$ before being joined by the vertex $v$ in the current step. (If $Q_{s,1}$ or $Q_{s,2}$ has no vertices, then we also let $T_{s,1}$ or $T_{s,2}$ be the null graph, i.e., the graph with empty vertex set.)
By induction we have
\begin{equation} \label{eq:v-T1-T2}
\begin{split}
  v(T_{s,1}) &\geq x_{s,j_{s,1}} \enspace, \\
  v(T_{s,2}) &\geq x_{s,j_{s,2}} \enspace.
\end{split}
\end{equation}
Combining our previous observations, we obtain that the vertex $v$ is contained in a tree $T$ satisfying
\begin{equation} \label{eq:vT-lb}
  v(T)= 1+\sum_{s\in[r]} \big(v(T_{s,1})+v(T_{s,2})\big)
      \geBy{eq:v-T1-T2} 1+\sum_{s\in[r]} (x_{s,j_{s,1}}+x_{s,j_{s,2}})
      \geByM{\eqref{eq:k-i},\eqref{eq:js12-sum}} k_i \enspace,
\end{equation}
where we also used Lemma~\ref{lemma:monotonicity} in the last step.
Combining \eqref{eq:x-sigma}, \eqref{eq:nu-sigma-d-sigma} and \eqref{eq:d-sigma-lambda} shows that the right hand side of \eqref{eq:vT-lb} equals $x_{\sigma,t}$, proving that the claimed invariant holds.
\end{proof}

We are now in a position to prove the lower bound in Proposition~\ref{prop:kstar-P1-Pr}. The argument is very similar to the inductive argument in the previous proof, but due to some subtleties we have to treat the step in which the game ends separately.

\begin{proof}[Proof of Proposition~\ref{prop:kstar-P1-Pr} (lower bound)]
We will argue that the Painter strategy $\AP_\alpha(P_{\ell_1},\ldots,P_{\ell_r})$ is a winning strategy in the $(P_{\ell_1},\ldots,P_{\ell_r})$-avoidance game with tree size restriction $k(\alpha)-1$, where $k(\alpha)$ is defined in \eqref{eq:recursion} and \eqref{eq:d-k-alpha}. Optimizing over the choice of $\alpha\in W(\ell_1,\ldots,\ell_r)$, we thus obtain a winning strategy for Painter in the game with tree size restriction $\max_{\alpha\in W(\ell_1,\ldots,\ell_r)} k(\alpha)-1$, as required.

Let $\alpha\in W(\ell_1,\ldots,\ell_r)$ be fixed and suppose Painter plays according to the strategy $\AP_\alpha(P_{\ell_1},\ldots,P_{\ell_r})$. Suppose for the sake of contradiction that Painter loses with a monochromatic path $P_{\ell_s}$ in some color $s\in[r]$. By the definition of Painter's strategy, this means that in the last step of the game assigning \emph{any} of the colors $s\in[r]$ to the last vertex $v$ would complete a path $P_{\ell_s}$ in color $s$. This implies that in each color $s\in[r]$ the vertex $v$ joins two (not necessarily maximal) paths $Q_{s,1}$ and $Q_{s,2}$ in color $s$ which for $j_{s,1}:=v(Q_{s,1})$ and $j_{s,2}:=v(Q_{s,2})$ satisfy
\begin{equation} \label{eq:js12-sum-ell}
  j_{s,1}+j_{s,2} = \ell_s-1 \enspace.
\end{equation}
Denoting for every $s\in[r]$ by $T_{s,1}$ and $T_{s,2}$ the components (=trees) that were joined by $v$ and that contain $Q_{s,1}$ and $Q_{s,2}$, respectively, we obtain from Lemma~\ref{lemma:strategy-invariant} that the vertex $v$ is contained in a tree $T$ satisfying
\begin{equation*}
  v(T)=1+\sum_{s\in[r]} \big(v(T_{s,1})+v(T_{s,2})\big)
      \geq 1+\sum_{s\in[r]} (x_{s,j_{s,1}}+x_{s,j_{s,2}})
      \geByM{\eqref{eq:k-i},\eqref{eq:k-alpha},\eqref{eq:js12-sum-ell}} k(\alpha) \enspace,
\end{equation*}
where in the last step we also used that for $d$ defined in \eqref{eq:d} we have $\nu_d=(\ell_1,\ldots,\ell_r)$.
This yields the desired contradiction and completes the proof.
\end{proof}

\subsection{Exact values of \texorpdfstring{$k^*(P_\ell,2)$}{k*(Pl,2)} for small values of \texorpdfstring{$\ell$}{l}}
\label{sec:implementation}

The values in Table~\ref{tab:kstar-Pell} were found by implementing the recursion in \eqref{eq:recursion} and \eqref{eq:d-k-alpha} and using Proposition~\ref{prop:kstar-P1-Pr}.
The computationally most expensive part in this approach is the maximization in \eqref{eq:kstar-P1-Pr}, as e.g.\ for the (symmetric) $P_\ell$-avoidance game with $r=2$ colors it requires maximizing over all strategy sequences from $W(\ell,\ell)$, of which there are $\binom{2(\ell-1)}{\ell-1}=4^{(1+o(1))\ell}$ many. However, by using an appropriate branch-and-bound technique, the set of strategy sequences to be considered in the maximization can be reduced substantially.
A program that implements this and further optimizations to compute $k^*(P_{\ell_1},P_{\ell_2})$ is available from the authors' websites~\cite{homepage-ms}.

We conclude this section by giving an example of a Painter strategy for the (symmetric) $P_\ell$-avoidance game with $r=2$ colors that outperforms the greedy strategy. For $\ell=28$, there are four strategy sequences from the set $W(28,28)$ achieving the optimal performance $k^*(P_{28},2)=28^2+7=791$ (cf.\ Table~\ref{tab:kstar-Pell}). They are given by $\alpha=(1)^6\circ(2,2)\circ(1)^7\circ(2)\circ(1)^{14}\circ(2)^{24}$, $\alpha'=(1,1,2,1,2,2)\circ(1)^{24}\circ(2)^{24}$, and the sequences $\ol{\alpha}$ and $\ol{\alpha}\,'$ that are obtained from $\alpha$ and $\alpha'$ by interchanging the $1$ and $2$ entries, exploiting the obvious symmetry. 

\section{Analyzing the recursion} \label{sec:analyze-recursion}

In this section we prove Theorems~\ref{thm:kstar-Pell-2}--\ref{thm:delta-c-bounds} by analyzing the recursion defined in \eqref{eq:recursion} and \eqref{eq:d-k-alpha} and using Proposition~\ref{prop:kstar-P1-Pr}. We focus on the asymmetric path-avoidance game in most of the upcoming arguments, and derive our results for the symmetric game at the very end. For the rest of this paper we restrict our attention to the case of $r=2$ colors.

\subsection{Asymptotic behaviour} \label{sec:asymptotics}

A crucial ingredient in our analysis of the recursion in \eqref{eq:recursion} and \eqref{eq:d-k-alpha} is the study of its asymptotic behaviour along a walk as described in Section~\ref{sec:defining-the-recursion} which after some initial turns moves towards infinity only along one coordinate direction (think e.g.\ of infinitely extending the walk in Figure~\ref{fig:recursion} in coordinate direction 1). The following completely self-contained lemma is the basis for this approach.

For any sequence $(x_\nu)_{\nu\geq 0}$ we define the corresponding sequence of first differences as $\Delta(x):=(x_{\nu+1}-x_\nu)_{\nu\geq 0}$.

\begin{lemma}[Recursion becomes periodic] \label{lemma:periodicity}
Let $x_0,\ldots,x_t$ and $\beta$ be arbitrary integers, and recursively define
\begin{equation} \label{eq:abstract-recursion}
  x_\nu:=\beta+\min_{\substack{j_1,j_2\geq 0: \\ j_1+j_2=\nu-1}}(x_{j_1}+x_{j_2}) \enspace, \qquad \nu\geq t+1\enspace.
\end{equation}
Furthermore, let $p$ be an integer from the set $\argmin_{0\leq j \leq t}\frac{x_j+\beta}{j+1}$.
Then the sequence $\Delta(x)=(x_{\nu+1}-x_\nu)_{\nu\geq 0}$ becomes periodic with period length $p+1$, and for all $\nu\geq t+1$ we have
\begin{equation} \label{eq:growth-upper-bound}
  x_\nu-x_{\nu-(p+1)}\leq x_p+\beta
\end{equation}
with equality for all large enough $\nu$. Moreover, for all $k\geq 1$ we have
\begin{equation} \label{eq:growth-lower-bound}
  x_{p+k(p+1)}-x_p\geq k(x_p+\beta) \enspace.
\end{equation}
\end{lemma}

Note that Lemma~\ref{lemma:periodicity} quantifies the asymptotic behaviour of the recursion~\eqref{eq:abstract-recursion}: In the long run, the values will change by $x_p+\beta$ every $p+1$ steps, i.e., for $\nu\to\infty$ we have
\begin{equation*}
  x_\nu=(\delta+o(1)) \cdot \nu \enspace,
\end{equation*}
where
\begin{equation} \label{eq:delta}
\delta=\delta(x_0,\ldots,x_t,\beta):=\min_{0\leq j \leq t}\frac{x_j+\beta}{j+1}\enspace.
\end{equation}

\begin{proof}
For any two integers $a\geq 1$ and $b$, applying the transformation
\begin{equation} \label{eq:substitution}
  y_\nu=a(x_\nu+\beta)-b(\nu+1)
\end{equation}
to \eqref{eq:abstract-recursion} yields an integer sequence $(y_\nu)_{\nu\geq 0}$ that satisfies the recursion
\begin{equation} \label{eq:recursion-subst}
  y_\nu=\min_{\substack{j_1,j_2\geq 0: \\ j_1+j_2=\nu-1}}(y_{j_1}+y_{j_2}) \enspace, \qquad \nu\geq t+1\enspace.
\end{equation}
Furthermore, by \eqref{eq:substitution} the first differences of the sequences $(x_\nu)_{\nu\geq 0}$ and $(y_\nu)_{\nu\geq 0}$ are related via
\begin{equation} \label{eq:delta-transform}
  \Delta(y)=a \Delta(x)-b \enspace.
\end{equation}
Applying the transformation~\eqref{eq:substitution} with
\begin{equation} \label{eq:choice-a-b}
  a:=p+1 \qquad \text{and} \qquad b:=x_p+\beta\enspace,
\end{equation}
we obtain with the definition of $p$ in the lemma that $y_p=0$ and $y_\nu\geq 0$ for all $0\leq \nu\leq t$. By these initial conditions and by \eqref{eq:recursion-subst}, \emph{all} elements of the sequence $(y_\nu)_{\nu\geq 0}$ are non-negative. Furthermore, using that $y_p=0$ it follows from \eqref{eq:recursion-subst} that
\begin{equation} \label{eq:decreasing-subsequence}
  y_\nu\leq y_{\nu-(p+1)} \quad \text{for all $\nu\geq t+1$} \enspace.
\end{equation}
Combining this with the non-negativity of the sequence $(y_\nu)_{\nu\geq 0}$, we obtain that for each residue class modulo $p+1$ the corresponding subsequence of $(y_\nu)_{\nu\geq 0}$ becomes constant, and that consequently the sequence itself becomes periodic with period length $p+1$. It follows that the sequence $\Delta(y)$ and by \eqref{eq:delta-transform} also the sequence $\Delta(x)$ become periodic with period length $p+1$.

Note that for all $\nu\geq t+1$ we have
\begin{equation*}
  x_\nu-x_{\nu-(p+1)} \eqBy{eq:delta-transform} \frac{y_\nu-y_{\nu-(p+1)}+b(p+1)}{a}
                      \leBy{eq:decreasing-subsequence} \frac{b(p+1)}{a}
                      \eqBy{eq:choice-a-b} x_p+\beta \enspace,
\end{equation*}
with equality for all large enough $\nu$, proving \eqref{eq:growth-upper-bound}.

Similarly, using the non-negativity of the sequence $(y_\nu)_{\nu\geq 0}$ and $y_p=0$ we obtain for all $k\geq 1$
\begin{equation*}
  x_{p+k(p+1)}-x_p \eqBy{eq:delta-transform} \frac{y_{p+k(p+1)}-y_p+bk(p+1)}{a} \geq \frac{bk(p+1)}{a} \eqBy{eq:choice-a-b} k(x_p+\beta) \enspace,
\end{equation*}
proving \eqref{eq:growth-lower-bound}.
\end{proof}

\subsection{Explicit version of Theorem~\texorpdfstring{\ref{thm:kstar-Pell-Pc}}{5}}

Using Lemma~\ref{lemma:periodicity} we will show that asymptotically optimal Painter strategies for the asymmetric $(P_\ell,P_c)$-avoidance game (i.e., strategies achieving the lower bound stated in Theorem~\ref{thm:kstar-Pell-Pc}) can be constructed as follows. Intuitively, we distinguish two phases of the corresponding walks: a short initial `preparation' phase and a long `payoff' phase, which is just a straight segment of the walk extending into coordinate direction $1$. The goal of the preparation phase is \emph{not} to directly optimize the resulting recursion values during this phase, but to optimize the constant $\delta$ as defined in~\eqref{eq:delta} that arises when applying Lemma~\ref{lemma:periodicity} to the payoff phase.

These ideas lead to the  following definition of the constant $\delta(c)$ appearing in Theorem~\ref{thm:kstar-Pell-Pc}. For any strategy sequence $\alpha\in W(\ell,c)$ we define
\begin{subequations} \label{eq:def-delta-c}
\begin{align}
  \beta(\alpha)  &:= 1+\min_{\substack{j_1,j_2\geq 0: \\ j_1+j_2=c-1}} (x_{2,j_1}+x_{2,j_2}) \enspace, \label{eq:beta-alpha} \\
  \delta(\alpha) &:= \min_{0\leq j\leq \ell-1} \frac{x_{1,j}+\beta(\alpha)}{j+1} \enspace, \label{eq:delta-alpha}
\end{align}
where $x_1=(x_{1,0},\ldots,x_{1,\ell-1})$ and $x_2=(x_{2,0},\ldots,x_{2,c-1})$ are defined via the recursion \eqref{eq:recursion}.
Using those definitions we set
\begin{equation} \label{eq:delta-1}
  \delta(1):=1 \enspace,
\end{equation}
and for any $c\geq 2$,
\begin{equation} \label{eq:delta-c}
  \delta(c):=\sup_{\substack{\ell\geq 1 \\ \alpha\in W(\ell,c) \colon \alpha_{\ell+c-2}=2}} \delta(\alpha) \enspace,
\end{equation}
\end{subequations}
where the condition $\alpha_{\ell+c-2}=2$ expresses that the last step of the walk corresponding to $\alpha$ is towards the second coordinate. (We will see in  Lemma~\ref{lemma:ub-delta-c} below that $\delta(c)$ is indeed a well-defined finite value.)

\begin{theorem}[Explicit version of Theorem~\ref{thm:kstar-Pell-Pc}] \label{thm:kstar-Pell-Pc-explicit}
For any $c\geq 1$ and any $\ell\geq 1$ we have
\begin{equation*}
  k^*(P_\ell,P_c)\leq \delta(c)\cdot\ell
\end{equation*}
and
\begin{equation*}
  k^*(P_\ell,P_c)\geq(\delta(c)-o(1))\cdot\ell \enspace,
\end{equation*}
where $\delta(c)$ is defined in \eqref{eq:def-delta-c}, 
and $o(1)$ stands for a non-negative function of $c$ and $\ell$ that tends to 0 for $c$ fixed and $\ell\rightarrow\infty$.
\end{theorem}

We prove Theorem \ref{thm:kstar-Pell-Pc-explicit} (and thus Theorem~\ref{thm:kstar-Pell-Pc}) in the next section.

\subsection{Proof of Theorem~\texorpdfstring{\ref{thm:kstar-Pell-Pc-explicit}}{16}} \label{sec:Pell-Pc-asymptotics}

We will prove the following three lemmas by induction. Note that Lemma~\ref{lemma:ub-k-ell-c-delta-c} is exactly the upper bound part of Theorem~\ref{thm:kstar-Pell-Pc-explicit}, and that moreover Lemma~\ref{lemma:ub-delta-c} yields the upper bound part of Theorem~\ref{thm:delta-c-bounds} (we have $\log_2(3)=1.584\ldots < 1.59$).

\begin{lemma}[Upper bound for $\delta(c)$] \label{lemma:ub-delta-c}
For any $c\geq 1$ we have $\delta(c) \leq c^{\log_2(3)}$.
\end{lemma}

\begin{lemma}[Monotonicity of $\delta(c)$] \label{lemma:delta-monotonicity}
For all $1\leq \chat\leq c$ we have $\delta(\chat\,)\leq \delta(c)$.
\end{lemma}

\begin{lemma}[Upper bound for $k^*(P_\ell,P_c)$ via $\delta(c)$] \label{lemma:ub-k-ell-c-delta-c}
For any $c\geq 1$ and any $\ell\geq 1$ we have $k^*(P_\ell,P_c)\leq \delta(c)\cdot\ell$.
\end{lemma}

\begin{proof}[Proof of Lemma~\ref{lemma:ub-delta-c}, \ref{lemma:delta-monotonicity} and \ref{lemma:ub-k-ell-c-delta-c}]

We argue by induction on $c$. For $c=1$ all claims are trivially satisfied. For the induction step let $c\geq 2$.

\textit{Induction step for Lemma~\ref{lemma:ub-delta-c}.}
For any fixed $\ell\geq 1$ consider an arbitrary fixed strategy sequence $\alpha\in W(\ell,c)$ with $\alpha_{\ell+c-2}=2$. Note that $\alpha$ can be uniquely written in the form
\begin{equation} \label{eq:alpha-parametrization}
  \alpha=(1)^{\ell_1-1}\circ (2)\circ (1)^{\ell_2-\ell_1}\circ (2)\circ \cdots \circ(1)^{\ell_{c-1}-\ell_{c-2}}\circ(2) \enspace,
\end{equation}
where $1\leq \ell_1\leq \ell_2\leq \cdots\leq\ell_{c-2}\leq \ell_{c-1}=\ell$.

In the following we derive upper bounds for the entries of the sequences $x_1$ and $x_2$ defined in \eqref{eq:recursion} for this $\alpha$.
For any $1\leq j\leq c-1$ we define $\alpha^{(j)}$ as the maximal prefix of $\alpha$ containing exactly $j-1$ entries equal to 2. By \eqref{eq:alpha-parametrization} we have $\alpha^{(j)}\in W(\ell_j,j)$. Moreover, by the definitions in \eqref{eq:recursion} and \eqref{eq:d-k-alpha} and by Proposition~\ref{prop:kstar-P1-Pr} we have
\begin{equation} \label{eq:x2j-kstar}
  x_{2,j}=k_{(\ell_j-1)+(j-1)}=k(\alpha^{(j)}) \leBy{eq:kstar-P1-Pr} k^*(P_{\ell_j},P_j) \enspace, \qquad 1\leq j\leq c-1 \enspace.
\end{equation}
By induction and Lemma~\ref{lemma:ub-k-ell-c-delta-c} we hence obtain from \eqref{eq:x2j-kstar} that
\begin{equation} \label{eq:x2j-deltaj}
  x_{2,j} \leq \delta(j)\cdot\ell_j \enspace, \quad 1\leq j\leq c-1 \enspace.
\end{equation}
Using that by \eqref{eq:recursion} and \eqref{eq:alpha-parametrization} the integers $x_{1,\ell_j-1}$ and $x_{2,j}$ correspond to sequence elements $k_i$ and $k_{i'}$ as defined in \eqref{eq:k-i} with $i<i'$, we obtain with Lemma~\ref{lemma:monotonicity} that
\begin{equation} \label{eq:x1ellj-deltaj}
  x_{1,\ell_j-1}\leq x_{2,j}-1 \leBy{eq:x2j-deltaj} \delta(j)\cdot\ell_j-1 \enspace, \quad 1\leq j\leq c-1 \enspace.
\end{equation}
By setting $j_1=\lfloor(c-1)/2\rfloor$ and  $j_2=\lceil(c-1)/2\rceil$ in \eqref{eq:beta-alpha} we obtain
\begin{equation} \label{eq:beta-alpha-x2}
  \beta(\alpha) \leBy{eq:beta-alpha} 1+x_{2,\lfloor(c-1)/2\rfloor}+x_{2,\lceil(c-1)/2\rceil} \leq 1+2x_{2,\lceil(c-1)/2\rceil} \leBy{eq:x2j-deltaj} 1+2\delta({\textstyle \left\lceil\frac{c-1}{2}\right\rceil})\cdot\ell_{\lceil(c-1)/2\rceil} \enspace,
\end{equation}
where we again used Lemma~\ref{lemma:monotonicity} in the second estimate.
Similarly, setting $j=\ell_{\lceil(c-1)/2\rceil}-1$ in \eqref{eq:delta-alpha} yields
\begin{equation} \label{eq:delta-alpha-ub}
  \delta(\alpha) \leBy{eq:delta-alpha} \frac{x_{1,\ell_{\lceil(c-1)/2\rceil}-1}+\beta(\alpha)}{\ell_{\lceil(c-1)/2\rceil}}
                 \leByM{\eqref{eq:x1ellj-deltaj},\eqref{eq:beta-alpha-x2}} 3\cdot\delta({\textstyle \left\lceil\frac{c-1}{2}\right\rceil}) \enspace.
\end{equation}
As the bound in \eqref{eq:delta-alpha-ub} holds for all $\ell\geq 1$ and all strategy sequences $\alpha\in W(\ell,c)$ with $\alpha_{\ell+c-2}=2$ simultaneously, we obtain with the definition in \eqref{eq:delta-c} that
\begin{equation*}
  \delta(c) \leq 3\cdot\delta({\textstyle \lceil\frac{c-1}{2}\rceil})
  \leq 3\cdot {\textstyle \lceil\frac{c-1}{2}\rceil}^{\log_2(3)}
  \leq 3\cdot \big({\textstyle \frac{c}{2}}\big)^{\log_2(3)} = c^{\log_2(3)} \enspace,
\end{equation*}
where the second estimate is the induction hypothesis. This completes the proof of Lemma~\ref{lemma:ub-delta-c}.

\textit{Induction step for Lemma~\ref{lemma:delta-monotonicity}.}
By induction we have $\delta(1)\leq \cdots \leq \delta(c-1)$, so it suffices to show that $\delta(c-1)\leq \delta(c)$ (by Lemma~\ref{lemma:ub-delta-c} we already know that $\delta(c)$ is a well-defined finite value). For $c=2$, note that the strategy sequence $\alpha=(2)\in W(1,2)$ yields $\beta(\alpha)=2$ and $\delta(\alpha)=2$ and consequently guarantees a lower bound of $\delta(2)\geq 2$, implying in particular that $\delta(1)\leq \delta(2)$ (recall~\eqref{eq:delta-1}).
For $c\geq 3$ we argue as follows: For each strategy sequence $\alpha^-\in W(\ell,c-1)$ with $\alpha_{\ell+(c-1)-2}^-=2$ consider the extended sequence $\alpha:=\alpha^-\circ(2) \in W(\ell,c)$. By Lemma~\ref{lemma:monotonicity} and \eqref{eq:beta-alpha} we have $\beta(\alpha^-)<\beta(\alpha)$, which by \eqref{eq:delta-alpha} implies that $\delta(\alpha^-)<\delta(\alpha)$. Using \eqref{eq:delta-c} this shows that $\delta(c-1)\leq\delta(c)$, completing the proof of Lemma~\ref{lemma:delta-monotonicity}.

\textit{Induction step for Lemma~\ref{lemma:ub-k-ell-c-delta-c}.}
For the reader's convenience, Figure~\ref{fig:subsequences} illustrates the notations used in this proof.

Let $\ell\geq 1$ and $\alpha\in W(\ell,c)$ be fixed. We show that for $k(\alpha)$ as defined in \eqref{eq:recursion} and \eqref{eq:d-k-alpha} we have $k(\alpha)\leq \delta(c)\cdot \ell$, from which the claim follows by Proposition~\ref{prop:kstar-P1-Pr}.
For the proof it is convenient to extend the sequences $x_1=(x_{1,0},\ldots,x_{1,\ell-1})$ and $x_2=(x_{2,0},\ldots,x_{2,c-1})$ defined in \eqref{eq:recursion} for the given $\alpha$ by setting 
\begin{equation} \label{eq:x1-extend}
  x_{1,\ell}:=k_d \eqBy{eq:k-alpha} k(\alpha)
\end{equation}
with $d=d(\ell,c)$ defined in \eqref{eq:d}.
Let $\ell'$ be such that $\alpha=\alpha'\circ(1)^{\ell-\ell'}$ with $\alpha'\in W(\ell',c)$ and $\alpha'_{\ell'+c-2}=2$.
Fixing some integer
\begin{equation} \label{eq:p-argmin}
  p\in \argmin_{0\leq j\leq \ell'-1} \frac{x_{1,j}+\beta(\alpha')}{j+1} \enspace,
\end{equation}
where $\beta(\alpha')$ is defined in \eqref{eq:beta-alpha}, we have
\begin{equation} \label{eq:growth-rate-bound}
  \frac{x_{1,p}+\beta(\alpha')}{p+1} \eqByM{\eqref{eq:delta-alpha},\eqref{eq:p-argmin}} \delta(\alpha') \leBy{eq:delta-c} \delta(c) \enspace.
\end{equation}
Let $\ellhat\leq \ell'-1$ be the largest integer such that
\begin{equation} \label{eq:ell-p+1}
  \ell = \ellhat+m(p+1)
\end{equation}
for some integer $m$.

\begin{figure}
\centering
\PSforPDF{
 \psfrag{col1}{Color 1}
 \psfrag{col2}{Color 2}
 \psfrag{1}{1}
 \psfrag{ell}{$\ell=\ellhat+m(p+1)$}
 \psfrag{ellp}{$\ell'$}
 \psfrag{ellh}{$\ellhat$}
 \psfrag{c}{$c$}
 \psfrag{ch}{$\chat$}
 \psfrag{p}{$p$}
 \psfrag{ppo}{$p+1$}
 \psfrag{alpha}{$\alpha\in W(\ell,c)$}
 \psfrag{alphap}{$\alpha'\in W(\ell',c)$}
 \psfrag{alphah}{$\alphahat\in W(\ellhat,\chat\,)$}
 \includegraphics{graphics/subsequences.eps}
}
\caption{Notations used in the proof of Lemma~\ref{lemma:ub-k-ell-c-delta-c}.} \label{fig:subsequences}
\end{figure}

By \eqref{eq:recursion}, \eqref{eq:x1-extend} and the definition of $\beta(\alpha')$ in \eqref{eq:beta-alpha} we have
\begin{equation*}
  x_{1,\nu}=\beta(\alpha')+\min_{\substack{j_1,j_2\geq 0: \\ j_1+j_2=\nu-1}}(x_{1,j_1}+x_{1,j_2}) \enspace,
  \qquad \ell'\leq\nu\leq\ell \enspace.
\end{equation*}
We may hence apply Lemma~\ref{lemma:periodicity}, and using \eqref{eq:p-argmin} we obtain that
\begin{equation} \label{eq:x1ell-x1p}
  x_{1,\ell} \eqBy{eq:ell-p+1} x_{1,\ellhat+m(p+1)} \leBy{eq:growth-upper-bound} x_{1,\ellhat}+m(x_{1,p}+\beta(\alpha')) \enspace.
\end{equation}

If $\ellhat\geq 1$, we let $\chat\,$ denote the maximal value of $\bar{c}$ for which $W(\ellhat,\bar{c})$ contains a prefix of $\alpha'$, and we let $\alphahat\in W(\ellhat,\chat\,)$ denote the corresponding prefix (see Figure~\ref{fig:subsequences}).
Clearly we have $\chat<c$ and
\begin{equation} \label{eq:k-alphahat-x1-ellhat}
  k(\alphahat) \eqByM{\eqref{eq:recursion},\eqref{eq:d-k-alpha}} x_{1,\ellhat} \enspace.
\end{equation}
As $\chat<c$ we may apply the induction hypothesis and obtain together with Proposition~\ref{prop:kstar-P1-Pr} that
\begin{equation*}
  k(\alphahat) \leBy{eq:kstar-P1-Pr} k^*(P_\ellhat\,,P_\chat) \leq \delta(\chat\,)\cdot \ellhat \enspace,
\end{equation*}
which combined with \eqref{eq:k-alphahat-x1-ellhat} and Lemma~\ref{lemma:delta-monotonicity} yields
\begin{equation} \label{eq:x-1-ellhat}
  x_{1,\ellhat} \leq \delta(c)\cdot\ellhat \enspace.
\end{equation}
If $\ellhat=0$, then \eqref{eq:x-1-ellhat} holds trivially (both sides of this inequality are equal to zero).

Combining our previous observations we obtain
\begin{equation*}
  k(\alpha) \leByM{\eqref{eq:x1-extend}, \eqref{eq:x1ell-x1p}} x_{1,\ellhat}+m(x_{1,p}+\beta(\alpha'))
       \leByM{\eqref{eq:growth-rate-bound}, \eqref{eq:x-1-ellhat}} \delta(c)\cdot (\ellhat+m(p+1))
       \eqBy{eq:ell-p+1} \delta(c)\cdot \ell \enspace,
\end{equation*}
completing the proof of Lemma~\ref{lemma:ub-k-ell-c-delta-c}.
\end{proof}

It remains to prove the lower bound in Theorem~\ref{thm:kstar-Pell-Pc-explicit}.

\begin{proof}[Proof of Theorem~\ref{thm:kstar-Pell-Pc-explicit} (lower bound)]
As in the proof of Lemma~\ref{lemma:ub-k-ell-c-delta-c} it is also convenient here to extend the definition in \eqref{eq:x-sigma} for any $\alpha\in W(\ell,c)$ by setting
\begin{equation} \label{eq:x1-extend'}
  x_{1,\ell}:=k_d \eqBy{eq:k-alpha} k(\alpha)
\end{equation}
with $d=d(\ell,c)$ defined in \eqref{eq:d}.
Note that the claim holds trivially if $c=1$, so we consider a fixed $c\geq 2$ in the following.
By the definition in \eqref{eq:delta-c} there are families $(\ell_t)_{t\geq 0}$ and $(\alpha^{(t)})_{t\geq 0}$, where $\alpha^{(t)}\in W(\ell_t,c)$ with $\alpha^{(t)}_{\ell_t+c-2}=2$, satisfying
\begin{equation} \label{eq:lim-delta-alpha-t}
  \lim_{t\rightarrow\infty} \delta(\alpha^{(t)})=\delta(c) \enspace.
\end{equation}
Fix any such strategy sequence $\alpha^{(t)}$, and for every $\ell\geq \ell_t$ consider the extended sequence $\alphahat^{(t)}:=\alpha^{(t)}\circ(1)^{\ell-\ell_t} \in W(\ell,c)$. Using \eqref{eq:beta-alpha} we obtain that for any such extended sequence $\alphahat^{(t)}$, the sequence $x_1=(x_{1,0},\ldots,x_{1,\ell})$ defined in \eqref{eq:recursion} and \eqref{eq:x1-extend'} satisfies
\begin{equation*}
  x_{1,\nu}=\beta(\alpha^{(t)})+\min_{\substack{j_1,j_2\geq 0: \\ j_1+j_2=\nu-1}}(x_{1,j_1}+x_{1,j_2}) \enspace,
  \qquad \ell_t\leq\nu\leq\ell \enspace.
\end{equation*}
Moreover, by \eqref{eq:x1-extend'} we have $k(\alphahat^{(t)})=x_{1,\ell}$. By the first part of Lemma~\ref{lemma:periodicity} and the definition in \eqref{eq:delta-alpha} we thus have
\begin{equation*}
  k(\alphahat^{(t)}) = (\delta(\alpha^{(t)})+o(1))\cdot\ell
\end{equation*}
for $c$ and $t$ fixed and $\ell\rightarrow\infty$ (recall that \eqref{eq:growth-upper-bound} holds with equality for all large enough indices).
Combining this with \eqref{eq:lim-delta-alpha-t} and applying Proposition~\ref{prop:kstar-P1-Pr} yields that
\begin{equation*}
  k^*(P_c, P_\ell) \geq (\delta(c)+o(1))\cdot\ell\enspace
\end{equation*}
for $c$ fixed and $\ell\to\infty$. Moreover, the upper bound given by Lemma~\ref{lemma:ub-k-ell-c-delta-c} shows that the term $o(1)$ must be non-positive.
\end{proof}

\subsection{Proof of Theorem~\texorpdfstring{\ref{thm:delta-1to6}}{6}} \label{sec:exact-values} 

In this section we derive the exact values of $\delta(c)$ stated in Theorem~\ref{thm:delta-1to6} by carrying out explicitly the optimization over lattice walks that appears in the definition \eqref{eq:delta-c}. Note that for small values of $c$, the walk corresponding to a strategy sequence $\alpha\in W(\ell,c)$ has only few turning points. We will derive upper bounds for the entries of the sequences $x_1$ and $x_2$ computed by the recursion \eqref{eq:recursion} as a function of the $1$-coordinates of these turning points. To do so we will only consider a few carefully selected terms when evaluating the minimization in~\eqref{eq:k-i}. The upper bound on $\delta(c)$ we obtain this way is the minimum over a small number of functions of $1$-coordinates of turning points, and turns out to be irrational. We will derive asymptotically matching lower bounds by describing families of strategy sequences for which the ratios between the $1$-coordinates of successive turning points approximate the optimal (irrational!) ratios.


We give the proof for $c=4$ here, and defer the (similar but more complicated) proof for $c\in\{5,6\}$ to the appendix.

\begin{proof}[Proof of $\delta(4)=\frac{1}{2}(\sqrt{13}+5)$]
Consider a strategy sequence $\alpha\in W(\ell,4)$ with $\alpha_{\ell+4-2}=2$, and note that $\alpha$ can be uniquely written in the form
\begin{equation*}
  \alpha=(1)^{\ell_1-1}\circ (2)\circ (1)^{\ell_2-\ell_1}\circ (2)\circ (1)^{\ell-\ell_2} \circ (2) \enspace,
\end{equation*}
where $1\leq \ell_1\leq \ell_2\leq \ell$.

Let $p\geq 1$ and $0\leq m<\ell_1$ be the unique integers satisfying
\begin{equation} \label{eq:para-ell2-c4}
  \ell_2=p \ell_1+m \enspace.
\end{equation}
The recursion \eqref{eq:recursion} yields by straightforward calculations that
\begin{equation} \label{eq:delta4-x1x2}
\begin{split}
  x_{1,j} &= j \enspace, \quad 0\leq j\leq \ell_1-1 \enspace, \qquad \text{(*)} \\
  x_{2,1} &= \ell_1 \enspace, \\
  x_{1,2\ell_1-1} &\leq x_{1,\ell_1-1}+x_{1,\ell_1-1}+x_{2,1}+1 = 3\ell_1-1 \enspace, \\
  x_{1,3\ell_1-1} &\leq x_{1,2\ell_1-1}+x_{1,\ell_1-1}+x_{2,1}+1 \leq 5\ell_1-1 \enspace, \\
  &\ldots \\
  x_{1,p \ell_1-1} &\leq x_{1,(p-1)\ell_1-1}+x_{1,\ell_1-1}+x_{2,1}+1 \leq 2p \ell_1-\ell_1-1 \enspace, \qquad \text{(**)} \\
  x_{1,\ell_2-1} &\leq x_{1,p \ell_1-1}+x_{1,m-1}+x_{2,1}+1 \leq 2p \ell_1+m-1 \eqBy{eq:para-ell2-c4} 2\ell_2-m-1 \enspace, \qquad \text{(***)} \\
  x_{2,2} &\leq  x_{1,p \ell_1-1}+x_{1,m}+x_{2,1}+1 \leq 2p \ell_1+m \eqBy{eq:para-ell2-c4} 2\ell_2-m \enspace,
\end{split}
\end{equation}
where a priori the inequality marked with~(**) holds only if $p\geq 2$ (as $x_{1,(p-1)\ell_1-1}$ is undefined otherwise). For $p=1$ the resulting inequality reads $x_{1,\ell_1-1}\leq\ell_1-1$, which is true nevertheless, as a comparison with~(*) shows. Similarly, a priori the inequality marked with~(***) holds only if $m\geq 1$ (as $x_{1,m-1}$ is undefined otherwise). For $m=0$ the resulting inequality reads $x_{1,\ell_2-1}\leq 2p\ell_1-1=2\ell_2-1$, which is true nevertheless, as a comparison with~(**) shows.

Defining
\begin{equation} \label{eq:mu}
  \mu:=m/\ell_1
\end{equation}
we thus obtain
\begin{align}
  \delta(\alpha) &\leByM{\eqref{eq:beta-alpha},\eqref{eq:delta-alpha}}
                  \min \left\{ \frac{x_{1,p \ell_1-1}+x_{2,1}+x_{2,2}+1}{p \ell_1},
                               \frac{x_{1,\ell_2-1}+x_{2,1}+x_{2,2}+1}{\ell_2} \right\} \notag \\
  &\leByM{\eqref{eq:delta4-x1x2},\eqref{eq:mu}} \min \left\{ 4+\frac{\mu}{p}, 4+\frac{1-2\mu}{p+\mu} \right\} \enspace. \label{eq:delta4-bound}
\end{align}
In order to determine the best bound resulting from this analysis, we need to find the maximum of the function on the right side of \eqref{eq:delta4-bound}. Relaxing this problem to a maximization problem with the integer variable $p\in\{1,2,\ldots\}$, and the \emph{real-valued} variable $0\leq \mu<1$ (cf.\ the definition in \eqref{eq:mu} and recall that $m\in\mathbb{N}$ is chosen such that $m<\ell_1$), it is easy to see that the right hand side of \eqref{eq:delta4-bound} attains its maximum for
\begin{equation} \label{eq:p-mu-sol}
  p=1 \qquad \text{and} \qquad \mu=\frac{1}{2}(\sqrt{13}-3) \enspace,
\end{equation}
corresponding to a ratio $\ell_2/\ell_1=p+\mu=\frac{1}{2}(\sqrt{13}-1)$ and yielding
\begin{equation} \label{eq:alpha4-bound}
  \delta(\alpha) \leByM{\eqref{eq:delta4-bound},\eqref{eq:p-mu-sol}} \frac{1}{2}(\sqrt{13}+5) \enspace.
\end{equation}
As the bound in \eqref{eq:alpha4-bound} holds for all $\ell\geq 1$ and all $\alpha\in W(\ell,4)$ with $\alpha_{\ell+4-2}=2$ simultaneously, we obtain with the definition in \eqref{eq:delta-c} that
\begin{equation*}
  \delta(4) \leq \frac{1}{2}(\sqrt{13}+5) \enspace.
\end{equation*}
To show that this upper bound is tight, by \eqref{eq:delta-c} it suffices to specify a family of strategy sequences $(\alpha^{(t)})_{t\geq 0}$, where $\alpha^{(t)}\in W(\ell_t,4)$ for some $\ell_t\geq 1$ and $\alpha_{\ell_t+4-2}^{(t)}=2$, with $\lim_{t\rightarrow\infty} \delta(\alpha^{(t)})=\frac{1}{2}(\sqrt{13}+5)$. Define
\begin{equation*}
  \alpha^{(t)}:=(1)^{\ell_{1,t}-1}\circ(2)\circ(1)^{\ell_{2,t}-\ell_{1,t}}\circ(2,2) \in W(\ell_{2,t},4)
\end{equation*}
for all $t\geq 0$, where $\ell_{1,t}:=10^t$ and $\ell_{2,t}:=\lfloor \frac{1}{2}(\sqrt{13}-1)\cdot 10^t\rfloor$, i.e., we have
\begin{equation} \label{eq:delta4-lim-ell-ratio}
  \lim_{t\rightarrow\infty} \frac{\ell_{2,t}}{\ell_{1,t}} = \frac{1}{2}(\sqrt{13}-1)\enspace.
\end{equation}
For each such strategy sequence $\alpha^{(t)}$ we obtain from \eqref{eq:recursion}, using that $\ell_{2,t}<2\ell_{1,t}$,
\begin{equation} \label{eq:delta4-x1x2-alpha-t}
\begin{split}
  x_{1,j} &= j \enspace, \quad 0\leq j\leq \ell_{1,t}-1 \enspace, \\
  x_{2,1} &= \ell_{1,t} \enspace, \\
  x_{1,j} &= \ell_{1,t}+j \enspace, \quad \ell_{1,t}\leq j\leq \ell_{2,t}-1 \enspace, \\
  x_{2,2} &= \ell_{1,t}+\ell_{2,t} \enspace, \\
  x_{2,3} &= 2\ell_{1,t}+\ell_{2,t} \enspace, \\
\end{split}
\end{equation}
implying that
\begin{equation} \label{eq:delta4-beta-alpha-t}
  \beta(\alpha^{(t)}) \eqByM{\eqref{eq:beta-alpha},\eqref{eq:delta4-x1x2-alpha-t}} 2\ell_{1,t}+\ell_{2,t}+1
\end{equation}
and
\begin{equation} \label{eq:delta4-delta-alpha-t}
  \delta(\alpha^{(t)}) \eqByM{\eqref{eq:delta-alpha},\eqref{eq:delta4-x1x2-alpha-t},\eqref{eq:delta4-beta-alpha-t}} \min \left\{ 3+\frac{\ell_{2,t}}{\ell_{1,t}} , 2+3\Big(\frac{\ell_{2,t}}{\ell_{1,t}}\Big)^{-1} \right\} \enspace.
\end{equation}
Using \eqref{eq:delta4-lim-ell-ratio} it follows from \eqref{eq:delta4-delta-alpha-t} that $\delta(\alpha^{(t)})\rightarrow\frac{1}{2}(\sqrt{13}+5)$ as $t\rightarrow\infty$.
\end{proof}

\subsection{Lower bound for \texorpdfstring{$\delta(c)$}{d(c)} via `bootstrapping'} \label{sec:lb-delta}

In the previous section, we determined asymptotically optimal strategy sequences for the asymmetric $(P_\ell, P_4)$-avoidance game. We now show how these can be `bootstrapped' to derive good strategy sequences for the asymmetric $(P_\ell, P_{c})$-avoidance game with any fixed $c$ of form $c=4^t$. 

Specifically, we construct these strategy sequences by nesting scaled copies of nearly-optimal strategy sequences for the $(P_\ell, P_4)$-avoidance game into each other. The lattice walks corresponding to these strategy sequences thus have a self-similar structure (see Figure~\ref{fig:bootstrapping}). As we will see, the equations arising in the analysis of this construction are, up to some error terms, exactly the same as in the proof that $\delta(4)=\frac{1}{2}(\sqrt{13}+5)$ in the previous section.

\begin{lemma}[Lower bound for $\delta(c)$ via `bootstrapping'] \label{lemma:bootstrapping}
For any integer $t\geq 0$, the function $\delta(c)$ defined in \eqref{eq:def-delta-c} satisfies
\begin{equation*}
  \delta(4^t)\geq \big({\delta(4)}\big)^t \;\eqByM{\text{Thm.~\ref{thm:delta-1to6}}}\; \left(\textstyle{\frac{1}{2}(\sqrt{13}+5)}\right)^t \enspace.
\end{equation*}
\end{lemma}

Note that together with the monotonicity guaranteed by Lemma~\ref{lemma:delta-monotonicity}, Lemma~\ref{lemma:bootstrapping} shows that $\delta(c)=\Omega(c^{\log_4(\delta(4))})=\Omega(c^{1.052\ldots})$ as a function of $c$, proving the bound claimed in Theorem~\ref{thm:delta-c-bounds}.

\begin{remark}
Similar lower bound statements can be proven by bootstrapping asymptotically optimal strategy sequences for the $(P_\ell,P_5)$- or the $(P_\ell,P_6)$-avoidance game. The resulting exponent of $c$ is marginally better for the case $(P_\ell,P_5)$ but worse for the case $(P_\ell,P_6)$: Using the values stated in Theorem~\ref{thm:delta-1to6} we obtain $\log_4(\delta(4))=\log_4(\frac{1}{2}(\sqrt{13}+5))=1.052\ldots$, $\log_5(\delta(5))=\log_5(\frac{1}{2}(\sqrt{24}+6))=1.053\ldots$, and $\log_6(\delta(6))=\log_6(\frac{1}{2}(\sqrt{37}+7))=1.048\ldots$.
\end{remark}

\begin{proof} [Proof of Lemma~\ref{lemma:bootstrapping}]

\begin{figure}
\centering
\PSforPDF{
 \psfrag{col1}{Color 1}
 \psfrag{col2}{Color 2}
 \psfrag{ctm1}{$c_{t-1}$}
 \psfrag{2ctm1}{$2c_{t-1}$}
 \psfrag{3ctm1}{$3c_{t-1}$}
 \psfrag{ct}{$c_t=4c_{t-1}$}
 \psfrag{l2tm1}{$\ell_{2,t-1}$}
 \psfrag{l1t}{$\ell_{1,t}=s\ell_{2,t-1}$}
 \psfrag{l2t}{$\ell_{2,t}=q\ell_{1,t}=sq\cdot\ell_{2,t-1}$}
 \psfrag{ldots}{$\ldots$}
 \psfrag{1}{$1$}
 \psfrag{alphatm1}{$\alpha^{(t-1)}$}
 \psfrag{alphat}{$\alpha^{(t)}\in W(\ell_{2,t},c_t)$}
 \psfrag{xo}{$x_{1,j} \,,\; 0\leq j\leq \ell_{1,t}-1$}
 \psfrag{x12}{$x_{1,j} \,,\; \ell_{1,t}\leq j\leq \ell_{2,t}-1$}
 \psfrag{x21}{$x_{2,j} \,,\; 0\leq j\leq c_{t-1}-1$}
 \psfrag{x22}{$x_{2,j} \,,\; c_{t-1}\leq j\leq 2c_{t-1}-1$}
 \psfrag{x23}{\parbox{3.5cm}{$x_{2,j}$ \\ $2c_{t-1}\leq j\leq 3c_{t-1}-1$}}
 \psfrag{x24}{\parbox{3.5cm}{$x_{2,j}$ \\ $3c_{t-1}\leq j\leq c_t-1$}}
 \includegraphics[scale=0.93]{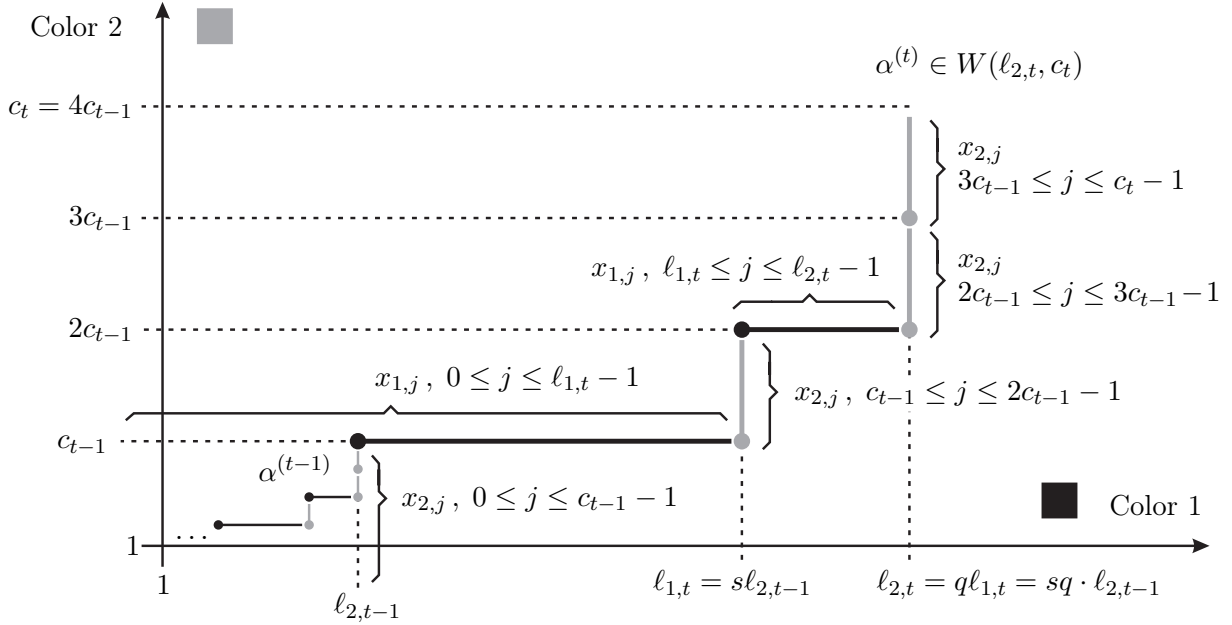}
}
\caption{Notations used in the proof of Lemma~\ref{lemma:bootstrapping}.} \label{fig:bootstrapping}
\end{figure}

In the following we specify a family of finite strategy sequences $(\alpha^{(t)})_{t\geq 0}$ such that $\alpha^{(t)}$ is a prefix of $\alpha^{(t+1)}$ for all $t\geq 0$. This defines an infinite sequence $\alpha$, and we denote by $x_1$ and $x_2$ the sequences defined in \eqref{eq:recursion} for this $\alpha$.

Let $q$ be a rational number with
\begin{equation*}
  1.3\leq q< \frac{1}{2}(\sqrt{13}-1)=1.302\ldots \enspace,
\end{equation*}
and let $s\geq 100$ be an integer such that $sq\in\NN$. We will consider these parameters fixed throughout the proof; at the very end we will take the limit $q\to \frac{1}{2}(\sqrt{13}-1)$ and $s\to\infty$.

We define
\begin{subequations} \label{eq:def-strategy}
\begin{equation}
  \ell_{1,0}:=1 \enspace, \quad \ell_{2,0}:=1 \enspace, \quad c_0:=1 \enspace, \quad \alpha^{(0)}:=() \in W(\ell_{2,0},c_0) \enspace,
\end{equation}
and for every $t\geq 1$,
\begin{equation} \label{eq:parameters-t}
  \ell_{1,t}:=s\ell_{2,t-1} \enspace, \quad
  \ell_{2,t}:=q\ell_{1,t}=s q \cdot\ell_{2,t-1} \enspace, \quad
  c_t:=4c_{t-1}=4^t
\end{equation}
and
\begin{equation}
  \alpha^{(t)}:=\alpha^{(t-1)}\circ(1)^{\ell_{1,t}-\ell_{2,t-1}}\circ(2)^{c_{t-1}}\circ(1)^{\ell_{2,t}-\ell_{1,t}}\circ(2)^{2c_{t-1}}\in W(\ell_{2,t},c_t) \enspace.
\end{equation}
\end{subequations}
For the reader's convenience those definitions are illustrated in Figure~\ref{fig:bootstrapping}.

By the definition in \eqref{eq:delta-alpha} we clearly have
\begin{equation} \label{eq:delta-alpha-0}
  \delta(\alpha^{(0)}) = 1 \enspace.
\end{equation}
We proceed by deriving lower bounds for $\delta(\alpha^{(t)})$, $t\geq 1$. For $t$ fixed, let
\begin{equation} \label{eq:p-argmin-t}
  p\in \argmin_{0\leq j\leq \ell_{2,t-1}-1} \frac{x_{1,j}+\beta(\alpha^{(t-1)})}{j+1} \enspace,
\end{equation}
where $\beta(\alpha^{(t-1)})$ is defined in \eqref{eq:beta-alpha}.
By Lemma~\ref{lemma:periodicity} we have
\begin{equation*}
  \frac{x_{1,p+k(p+1)}-x_{1,p}}{k(p+1)} \geBy{eq:growth-lower-bound} \frac{x_{1,p}+\beta(\alpha^{(t-1)})}{p+1}
                                        \eqByM{\eqref{eq:delta-alpha},\eqref{eq:p-argmin-t}} \delta(\alpha^{(t-1)})
\end{equation*}
for all $k\geq 1$ with $p+k(p+1)\leq \ell_{1,t}-1$. As the entries of the sequence $x_1$ are non-negative and increasing (recall Lemma~\ref{lemma:monotonicity}) this implies
\begin{subequations} \label{eq:x1x2-alpha-t}
\begin{equation} \label{eq:x1j-lb}
  x_{1,j} \geq \delta(\alpha^{(t-1)})\cdot(j-2p)\geq \delta(\alpha^{(t-1)})\cdot(j-2\ell_{2,t-1}+2) \enspace, \quad 0\leq j\leq \ell_{1,t}-1 \enspace,
\end{equation}
where we used that $p\leq \ell_{2,t-1}-1$ in the second estimate.
Using \eqref{eq:x1j-lb} and the trivial bound
\begin{equation}
  x_{2,j} \geq 0 \enspace, \quad 0\leq j\leq c_{t-1}-1 \enspace,
\end{equation}
we obtain from \eqref{eq:recursion}, using that $\ell_{2,t}<2\ell_{1,t}$,
\begin{equation}
\begin{split}
  x_{2,j} &\geq \delta(\alpha^{(t-1)})\cdot(\ell_{1,t}-4\ell_{2,t-1}+3) \enspace, \quad c_{t-1}\leq j\leq 2c_{t-1}-1 \enspace, \\
  x_{1,j} &\geq \delta(\alpha^{(t-1)})\cdot(\ell_{1,t}+j-8\ell_{2,t-1}+6) \enspace, \quad \ell_{1,t}\leq j\leq \ell_{2,t}-1 \enspace, \\
  x_{2,j} &\geq \delta(\alpha^{(t-1)})\cdot(\ell_{1,t}+\ell_{2,t}-8\ell_{2,t-1}+6) \enspace, \quad 2c_{t-1}\leq j\leq 3c_{t-1}-1 \enspace, \\
  x_{2,j} &\geq \delta(\alpha^{(t-1)})\cdot(2\ell_{1,t}+\ell_{2,t}-12\ell_{2,t-1}+9) \enspace, \quad 3c_{t-1}\leq j\leq 4c_{t-1}-1 \enspace, \\
\end{split}
\end{equation}
\end{subequations}
where we ignored the summand $+1$ arising from \eqref{eq:k-i} in all these lower bounds.

It follows that
\begin{equation} \label{eq:beta-alpha-t}
  \beta(\alpha^{(t)}) \geByM{\eqref{eq:beta-alpha},\eqref{eq:x1x2-alpha-t}} \delta(\alpha^{(t-1)})(2\ell_{1,t}+\ell_{2,t}-12\ell_{2,t-1})
\end{equation}
(where we ignored the summand $+1$ from \eqref{eq:beta-alpha} and the summand $+9\delta(\alpha^{(t-1)})$ from \eqref{eq:x1x2-alpha-t}).
It is readily checked that for the lower bounds given in \eqref{eq:x1x2-alpha-t} and \eqref{eq:beta-alpha-t}, the minimum in \eqref{eq:delta-alpha} is attained either for $j=\ell_{1,t}-1$ or for $j=\ell_{2,t}-1$, yielding
\begin{equation} \label{eq:delta-alpha-t}
\begin{split}
  \delta(\alpha^{(t)}) &\geByM{\eqref{eq:delta-alpha},\eqref{eq:x1x2-alpha-t},\eqref{eq:beta-alpha-t}}
  \delta(\alpha^{(t-1)})\cdot \min \left\{\frac{3\ell_{1,t}+\ell_{2,t}-14\ell_{2,t-1}+1}{\ell_{1,t}}, \frac{3\ell_{1,t}+2\ell_{2,t}-20\ell_{2,t-1}+5}{\ell_{2,t}} \right\} \\
  &\geBy{eq:parameters-t}
  \delta(\alpha^{(t-1)})\cdot \underbrace{\min \left\{ 3+q-\frac{14}{s}, 2+3q^{-1}-\frac{20}{qs} \right\}}_{=:f(q,s)}
  =\delta(\alpha^{(t-1)})\cdot f(q,s) \enspace.
  \end{split}
\end{equation}
(Note the similarities between \eqref{eq:delta4-x1x2-alpha-t}, \eqref{eq:delta4-beta-alpha-t}, \eqref{eq:delta4-delta-alpha-t} and \eqref{eq:x1x2-alpha-t}, \eqref{eq:beta-alpha-t}, \eqref{eq:delta-alpha-t}, respectively.)

Combining \eqref{eq:delta-alpha-0} and \eqref{eq:delta-alpha-t} and using the definition in \eqref {eq:delta-c} we thus have 
\begin{equation} \label{eq:delta-a-t}
  \delta(c_t)\geq\delta(\alpha^{(t)}) \geq \big( f(q,s)\big)^t
\end{equation}
for all $t\geq 0$.
Observing that for $q\to \frac{1}{2}(\sqrt{13}-1)$ and $s\to\infty$ we have
\begin{equation*}
  f(q,s)\to \frac{1}{2}(\sqrt{13}+5)=\delta(4)
\end{equation*}
(recall Theorem~\ref{thm:delta-1to6}), we obtain that for any $t\geq 0$ we have
\begin{equation*}
  \delta(4^t)=\delta(c_t)\geq \big(\delta(4)\big)^t \enspace.
\end{equation*}
\end{proof}

\subsection{Lower bound for \texorpdfstring{$k^*(P_\ell,2)$}{k*(Pl,2)}}
\label{sec:lb-symmetric}

An extension of the construction in the previous section finally yields the lower bound on $k^*(P_\ell,2)$ claimed in Theorem~\ref{thm:kstar-Pell-2}.

\begin{proof}[Proof of Theorem~\ref{thm:kstar-Pell-2} (lower bound)]
We will reuse most of the analysis from the previous proof for the fixed parameters
\begin{equation} \label{eq:values-q-s}
  q:=1.3 \enspace, \qquad s:=320 \enspace.
\end{equation}

Note that for $q$ and $s$ as in $\eqref{eq:values-q-s}$ and any $t\geq 0$, the analysis of the strategy sequence $\alpha^{(t)} \in W(\ell_{2,t},c_t)$ defined in~\eqref{eq:def-strategy} yields that
\begin{equation} \label{eq:fqs-concrete}
  \delta(\alpha^{(t)}) \geBy{eq:delta-a-t} \big( f(q,s)\big)^t\geq \left(\frac{17}{4}\right)^t\enspace.
\end{equation}

For $t\geq 0$, we now set
\begin{equation} \label{eq:ellhat-t}
  \ellhat_t:=10\ell_{2,t} \eqBy{eq:parameters-t} 10\cdot (sq)^t \eqBy{eq:values-q-s} 10\cdot 416^t
\end{equation}
and extend $\alpha^{(t)}\in W(\ell_{2,t},c_t)$ to a strategy sequence $\alphahat^{(t)}\in W(\ellhat_t,\ellhat_t)$ by defining
\begin{equation*}
  \alphahat^{(t)}:=\alpha^{(t)}\circ (1)^{\ellhat_t-\ell_{2,t}} \circ (2)^{\ellhat_t-c_t} \enspace.
\end{equation*}

Similarly to the previous proof, we obtain that the sequences $x_1$ and $x_2$ defined for $\alphahat^{(t)}$ in~\eqref{eq:recursion} satisfy
\begin{equation*}
  x_{1,j} \geq \delta(\alpha^{(t)})\cdot(j-2\ell_{2,t}+2) \enspace, \quad 0\leq j\leq \ellhat_t-1 \enspace, \\
\end{equation*}
and
\begin{equation*}
  x_{2,j} \geq \delta(\alpha^{(t)})\cdot (k-1)(\ellhat_t-4\ell_{2,t}+3) \enspace, \quad (k-1)c_t\leq j\leq kc_t-1 \enspace,
\end{equation*}
for every integer $1\leq k\leq \ellhat_t/c_t$. Putting everything together, we obtain
\begin{align}
  k(\alphahat^{(t)}) &\geq \delta(\alpha^{(t)})\cdot \frac{\ellhat_t}{c_t} \cdot \Big(\ellhat_t-4\ell_{2,t}+3\Big) \notag \\
                &\geq  \Big(1-4\frac{\ell_{2,t}}{\ellhat_t}\Big) \cdot \frac{\delta(\alpha^{(t)})}{c_t} \cdot \ellhat_t\,\!^2 \notag \\
                &\geByM{\eqref{eq:parameters-t},\eqref{eq:fqs-concrete},\eqref{eq:ellhat-t}}  0.6 \cdot \Big(\frac{17}{16}\Big)^t \cdot \ellhat_t\,\!^2 \notag \\
                &\geBy{eq:ellhat-t}0.6 \cdot \Big(\frac{\ellhat_t}{10}\Big)^{\log_{416}(17/16)} \cdot \ellhat_t\,\!^2
                \geq 0.5 \cdot \ellhat_t\,\!^{2.01} \label{eq:k-alpha-t+} \enspace.
\end{align}

Applying Proposition~\ref{prop:kstar-P1-Pr} we obtain from \eqref{eq:k-alpha-t+} that $k^*(P_\ell,2)=k^*(P_\ell,P_\ell)=\Omega(\ell^{2.01})$, proving the claimed lower bound.
\end{proof}

\begin{remark}
Observe that when performing the analysis of the previous proof with the scaling factor $s$ as a variable (and $q$ near $\frac{1}{2}(\sqrt{13}-1)$ fixed), we obtain an improvement of
\begin{equation*}
  \log_{sq} \big(f(s,q)/4\big)
\end{equation*}
over $2$ in the exponent of $\ellhat_t$ in \eqref{eq:k-alpha-t+}. The choice of $s:=320$ in~\eqref{eq:values-q-s} roughly maximizes this gain.
\end{remark}

\subsection{Putting everything together}

In this last section we complete the proofs of Theorem~\ref{thm:delta-c-bounds} and Theorem~\ref{thm:kstar-Pell-2} by collecting our findings from throughout the paper. 

\begin{proof}[Proof of Theorem~\ref{thm:delta-c-bounds}]
As already mentioned, the upper bound follows by combining Theorem~\ref{thm:kstar-Pell-Pc-explicit} with Lemma~\ref{lemma:ub-delta-c}, observing that $\log_2(3)=1.584\ldots <1.59$. The lower bound follows by combining Theorem~\ref{thm:kstar-Pell-Pc-explicit} with Lemma~\ref{lemma:bootstrapping}, using the monotonicity guaranteed by Lemma~\ref{lemma:delta-monotonicity} and observing that $\log_4(\delta(4))=\log_4(\frac{1}{2}(\sqrt{13}+5))=1.052\ldots >1.05$.
\end{proof}

\begin{proof}[Proof of Theorem~\ref{thm:kstar-Pell-2}]
As already mentioned, the upper bound follows immediately by combining Theorem~\ref{thm:kstar-Pell-Pc} with the upper bound on $\delta(c)$ stated in Theorem~\ref{thm:delta-c-bounds}, using the non-negativity of the $o(1)$ term in Theorem~\ref{thm:kstar-Pell-Pc}. The proof of the lower bound was given in Section~\ref{sec:lb-symmetric}.
\end{proof}

\bibliographystyle{plain}
\bibliography{refs}

\begin{thebibliography}{1}

\bibitem{homepage-ms}
currently \url{http://www.as.inf.ethz.ch/muetze} and
  \url{http://www.mpi-inf.mpg.de/~rspoehel}.

\bibitem{org-new}
M.~Belfrage, T.~M{\"u}tze, and R.~Sp{\"o}hel.
\newblock Probabilistic one-player {R}amsey games via deterministic two-player
  games.
\newblock Submitted.

\bibitem{MR1248487}
A.~Kurek and A.~Ruci{\'n}ski.
\newblock Globally sparse vertex-{R}amsey graphs.
\newblock {\em J. Graph Theory}, 18(1):73--81, 1994.

\bibitem{MR2152058}
A.~Kurek and A.~Ruci{\'n}ski.
\newblock Two variants of the size {R}amsey number.
\newblock {\em Discuss. Math. Graph Theory}, 25(1-2):141--149, 2005.

\bibitem{MR1182457}
T.~{\L}uczak, A.~Ruci{\'n}ski, and B.~Voigt.
\newblock Ramsey properties of random graphs.
\newblock {\em J. Combin. Theory Ser. B}, 56(1):55--68, 1992.

\bibitem{ovg-combinatorica}
M.~Marciniszyn and R.~Spöhel.
\newblock Online vertex-coloring games in random graphs.
\newblock {\em Combinatorica}, 30(1):105--123, 2010.
\newblock An extended abstract appeared in the proceedings of SODA '07.

\bibitem{org-lb}
M.~Marciniszyn, R.~Sp{\"o}hel, and A.~Steger.
\newblock Online {R}amsey games in random graphs.
\newblock {\em Combin. Probab. Comput.}, 18(Special Issue 1-2):271--300, 2009.

\bibitem{mrs11}
T.~Mütze, T.~Rast, and R.~Spöhel.
\newblock Coloring random graphs online without creating monochromatic
  subgraphs.
\newblock In {\em Proceedings of the 22nd annual ACM-SIAM Symposium on Discrete
  Algorithms (SODA '11)}, pages 145--158, 2011.

\bibitem{Vygen201167}
J.~Vygen.
\newblock Splitting trees at vertices.
\newblock {\em Discrete Mathematics}, 311(1):67--69, 2011.

\end{thebibliography}

\appendix
\section*{Appendix}

We provide the missing part of the proof of Theorem~\ref{thm:delta-1to6} (see Section~\ref{sec:exact-values}).

\begin{proof}[Proof of $\delta(5)=\sqrt{6}+3=\frac{1}{2}(\sqrt{24}+6)$ and $\delta(6)=\frac{1}{2}(\sqrt{37}+7)$]
The proof is similar to the proof of $\delta(4)=\frac{1}{2}(\sqrt{13}+5)$ given in Section~\ref{sec:exact-values}, but involves some slightly more technical calculations.

Consider a strategy sequence $\alpha\in W(\ell,c)$ with $c\in\{5,6\}$ and $\alpha_{\ell+c-2}=2$, and note that $\alpha$ can be uniquely written in the form
\begin{equation*}
  \alpha=(1)^{\ell_1-1}\circ (2)\circ (1)^{\ell_2-\ell_1}\circ (2)\circ (1)^{\ell_3-\ell_2}\circ (2)\circ \alpha' \enspace,
\end{equation*}
where $1\leq \ell_1\leq \ell_2\leq \ell_3\leq \ell$.

Let $p_2\geq 1$ and $0\leq m_2<\ell_1$ be the unique integers satisfying
\begin{equation*}
  \ell_2=p_2 \ell_1+m_2 \enspace.
\end{equation*}

If $p_2=1$ then let $p_3\geq 1$ and $0\leq m_3<\ell_1$ denote the unique integers satisfying
\begin{equation} \label{eq:para-ell3-1}
  \ell_3=p_3 \ell_1+m_3
\end{equation}
(note that $p_3=1$ implies that $m_3\geq m_2$).
 
If $p_2\geq 2$ then let $q_3\geq 1$, $0\leq p_3<p_2$ and $0\leq m_3<\ell_1$ denote the unique integers satisfying
\begin{equation} \label{eq:para-ell3-2}
  \ell_3=q_3 p_2 \ell_1+p_3 \ell_1+m_3
\end{equation}
(note that $q_3=1$ and $p_3=0$ implies that $m_3\geq m_2$).

If $p_2=1$, then we obtain from \eqref{eq:recursion}, reusing the results from \eqref{eq:delta4-x1x2},
\begin{equation} \label{eq:x1x2-p2-eq-1}
\begin{split}
  x_{1,j} &= j \enspace, \quad 0\leq j\leq \ell_1-1 \enspace, \qquad \text{(*)} \\
  x_{2,1} &= \ell_1 \enspace, \\
  x_{1,\ell_2-1} &\leq 2\ell_2-m_2-1 \\
  x_{2,2} &\leq 2\ell_1+m_2 = 2\ell_2-m_2 \enspace, \\
  x_{1,\ell_1+m_3-1} &\leq x_{1,\ell_1-1}+x_{1,m_3-1}+(2-\indic{m_2\geq m_3})x_{2,1}+1 = (3-\indic{m_2\geq m_3})\ell_1+m_3-1 \enspace, \quad \text{(**)} \\
  x_{1,\ell_1+m_3} &\leq x_{1,\ell_1-1}+x_{1,m_3}+(2-\indic{m_2>m_3})x_{2,1}+1 = (3-\indic{m_2>m_3})\ell_1+m_3 \enspace, \\
  x_{1,2\ell_1-1} &\leq x_{1,\ell_1-1}+x_{1,\ell_1-1}+2x_{2,1}+1 = 4\ell_1-1 \enspace, \\
  x_{1,2\ell_1+m_3-1} &\leq x_{1,\ell_1+m_3-1}+x_{1,\ell_1-1}+2x_{2,1}+1 \leq (6-\indic{m_2\geq m_3})\ell_1+m_3-1 \enspace, \\
  x_{1,2\ell_1+m_3} &\leq x_{1,\ell_1+m_3}+x_{1,\ell_1-1}+2x_{2,1}+1 \leq (6-\indic{m_2>m_3})\ell_1+m_3 \enspace, \\
  x_{1,3\ell_1-1} &\leq x_{1,2\ell_1-1}+x_{1,\ell_1-1}+2x_{2,1}+1 \leq 7\ell_1-1 \enspace, \\
  x_{1,3\ell_1+m_3-1} &\leq x_{1,2\ell_1+m_3-1}+x_{1,\ell_1-1}+2x_{2,1}+1 \leq (9-\indic{m_2\geq m_3})\ell_1+m_3-1 \enspace, \\
  x_{1,3\ell_1+m_3} &\leq x_{1,2\ell_1+m_3}+x_{1,\ell_1-1}+2x_{2,1}+1 \leq (9-\indic{m_2>m_3})\ell_1+m_3 \enspace, \\
  &\ldots \\
  x_{1,p_3 \ell_1-1} &\leq x_{1,(p_3-1)\ell_1-1}+x_{1,\ell_1-1}+2x_{2,1}+1 \leq (3p_3-2)\ell_1-1 \enspace, \quad \text{(***)} \\
  x_{1,\ell_3-1} &\leq x_{1,(p_3-1)\ell_1+m_3-1}+x_{1,\ell_1-1}+2x_{2,1}+1 \leq (3p_3-\indic{m_2\geq m_3})\ell_1+m_3-1 \\
                 &\hspace{67mm}\eqBy{eq:para-ell3-1} 3\ell_3-2m_3-\indic{m_2\geq m_3}\ell_1-1 \enspace, \quad \text{(****)} \\
  x_{2,3} &\leq x_{1,(p_3-1)\ell_1+m_3}+x_{1,\ell_1-1}+2x_{2,1}+1  \leq (3p_3-\indic{m_2>m_3})\ell_1+m_3 \\
                 &\hspace{67mm}\eqBy{eq:para-ell3-1} 3\ell_3-2m_3-\indic{m_2>m_3}\ell_1 \enspace.
\end{split}
\end{equation}
A priori the inequality marked with (**) holds only if $m_3\geq 1$ (as $x_{1,m_3-1}$ is undefined otherwise), but for $m_3=0$ the resulting inequality is true nevertheless, as a comparison with (*) shows. Similarly, the inequality resulting from~(***) holds even if $p_3=1$, and the inequality resulting from~(****) holds even if $p_3=1$ and $m_3=0$.

Defining
\begin{equation} \label{eq:mu23}
  \mu_2:=m_2/\ell_1 \qquad \text{and} \qquad \mu_3:=m_3/\ell_1
\end{equation}
and combining the results from \eqref{eq:x1x2-p2-eq-1} yields
\begin{subequations} \label{eq:bounds-delta5-p2-eq-1}
\begin{align}
  \frac{x_{1,\ell_1-1}+2x_{2,2}+1}{\ell_1} &\leq 5+2\mu_2 \enspace, \label{eq:ell-5-b11} \\
  \frac{x_{1,\ell_2-1}+2x_{2,2}+1}{\ell_2} &\leq 6-\frac{3\mu_2}{1+\mu_2} \enspace, \label{eq:ell-5-b12} \\
  \frac{x_{1,p_3 \ell_1-1}+2x_{2,2}+1}{p_3 \ell_1} &\leq 3+\frac{2+2\mu_2}{p_3} \enspace, \label{eq:ell-5-b13} \\
  \frac{x_{1,p_3 \ell_1-1}+x_{2,1}+x_{2,3}+1}{p_3 \ell_1} &\leq 6+\frac{\mu_3-1-\indic{\mu_2>\mu_3}}{p_3} \enspace, \label{eq:ell-5-b14} \\
  \frac{x_{1,\ell_3-1}+x_{2,1}+x_{2,3}+1}{\ell_3} &\leq 6+\frac{1-4\mu_3-\indic{\mu_2\geq \mu_3}-\indic{\mu_2>\mu_3}}{p_3+\mu_3} \label{eq:ell-5-b15}
\end{align}
\end{subequations}
and
\begin{subequations} \label{eq:bounds-delta6-p2-eq-1}
\begin{align}
  \frac{x_{1,\ell_2-1}+x_{2,2}+x_{2,3}+1}{\ell_2} &\leq 4+\frac{3p_3-2\mu_2+\mu_3-\indic{\mu_2>\mu_3}}{1+\mu_2} \enspace, \label{eq:ell-6-b12} \\
  \frac{x_{1,p_3 \ell_1-1}+x_{2,2}+x_{2,3}+1}{p_3 \ell_1} &\leq 6+\frac{\mu_2+\mu_3-\indic{\mu_2>\mu_3}}{p_3} \enspace, \label{eq:ell-6-b13} \\
  \frac{x_{1,\ell_3-1}+x_{2,2}+x_{2,3}+1}{\ell_3} &\leq 6+\frac{2+\mu_2-4\mu_3-\indic{\mu_2\geq\mu_3}-\indic{\mu_2>\mu_3}}{p_3+\mu_3} \enspace. \label{eq:ell-6-b14}
\end{align}
\end{subequations}

If $p_2\geq 2$, then we obtain from \eqref{eq:recursion}, reusing the results from \eqref{eq:delta4-x1x2},
\begin{subequations} \label{eq:x1x2-p2-geq-2}
\begin{equation}
\begin{split}
  x_{1,j} &= j \enspace, \quad 0\leq j\leq \ell_1-1 \enspace, \qquad \text{(*)} \\
  x_{2,1} &= \ell_1 \enspace, \\
  x_{1,2\ell_1-1} &\leq 3\ell_1-1 \enspace, \\
  x_{1,3\ell_1-1} &\leq 5\ell_1-1 \enspace, \\
  &\ldots \\
  x_{1,p_3 \ell_1-1} &\leq 2p_3\ell_1-\ell_1-1 \enspace, \qquad \text{(**)} \\
  x_{1,p_3 \ell_1+m_3-1} &\leq x_{1,p_3 \ell_1-1}+x_{1,m_3-1}+x_{2,1}+1 \leq 2p_3 \ell_1+m_3-1 \enspace, \qquad \text{(***)} \\  
  x_{1,p_3 \ell_1+m_3} &\leq x_{1,p_3 \ell_1-1}+x_{1,m_3}+x_{2,1}+1 \leq 2p_3 \ell_1+m_3 \enspace, \qquad \text{(****)} \\  
  x_{1,p_2 \ell_1-1} &\leq 2p_2 \ell_1-\ell_1-1 \enspace, \qquad \text{(+)} \\
  x_{1,\ell_2-1} &\leq 2\ell_2-m_2-1 \enspace, \\
  x_{2,2} &\leq 2p_2 \ell_1+m_2 \leq 2\ell_2-m_2 \enspace, \\
  x_{1,2p_2 \ell_1-1} &\leq 2x_{1,p_2 \ell_1-1}+2x_{2,1}+1 \leq 4p_2 \ell_1-1 \enspace, \\
  x_{1,3p_2 \ell_1-1} &\leq x_{1,2p_2 \ell_1-1}+x_{1,p_2 \ell_1-1}+2x_{2,1}+1 \leq 6p_2 \ell_1+\ell_1-1 \enspace, \\
  x_{1,4p_2 \ell_1-1} &\leq x_{1,3p_2 \ell_1-1}+x_{1,p_2 \ell_1-1}+2x_{2,1}+1 \leq 8p_2 \ell_1+2\ell_1-1 \enspace, \hspace{33mm} \\
  &\ldots \\
\end{split}
\end{equation}
\begin{equation} 
\begin{split}
  x_{1,q_3 p_2 \ell_1-1} &\leq x_{1,(q_3-1)p_2 \ell_1-1}+x_{1,p_2 \ell_1-1}+2x_{2,1}+1 \leq 2q_3 p_2 \ell_1+(q_3-2)\ell_1-1 \enspace, \quad \text{(++)} \\
  x_{1,q_3 p_2 \ell_1+\ell_1-1} &\leq x_{1,q_3 p_2 \ell_1-1}+x_{1,\ell_1-1}+2x_{2,1}+1 \leq 2q_3 p_2 \ell_1+(q_3+1)\ell_1-1 \enspace, \\
  x_{1,q_3 p_2 \ell_1+2\ell_1-1} &\leq x_{1,q_3 p_2 \ell_1-1}+x_{1,2\ell_1-1}+2x_{2,1}+1 \leq 2q_3 p_2 \ell_1+(q_3+3)\ell_1-1 \enspace, \\
  &\ldots \\
  x_{1,q_3 p_2 \ell_1+p_3 \ell_1-1} &\leq x_{1,q_3 p_2 \ell_1-1}+x_{1,p_3 \ell_1-1}+2x_{2,1}+1 \leq 2(q_3 p_2 \ell_1+p_3 \ell_1)+(q_3-1)\ell_1-1 \enspace, \quad \text{(+++)} \\
  x_{1,\ell_3-1} &\leq x_{1,q_3 p_2 \ell_1-1}+x_{1,p_3 \ell_1+m_3-1}+2x_{2,1}+1 \leq 2(q_3 p_2 \ell_1+p_3 \ell_1)+q_3 \ell_1+m_3 - 1 \\
                 &\hspace{67mm}\eqBy{eq:para-ell3-2} 2 \ell_3+q_3 \ell_1-m_3-1 \enspace, \quad \text{(++++)} \\
  x_{2,3} &\leq x_{1,q_3 p_2 \ell_1-1}+x_{1,p_3 \ell_1+m_3}+2x_{2,1}+1 \leq 2(q_3 p_2 \ell_1+p_3 \ell_1)+q_3 \ell_1+m_3 \\
          &\hspace{67mm}\eqBy{eq:para-ell3-2} 2 \ell_3+q_3 \ell_1-m_3 \enspace.
\end{split}
\end{equation}
\end{subequations}
Note that the inequalities resulting from~(***) and (****) hold even if $p_3=0$ (in this case $x_{1,p_3 \ell_1-1}$ is not defined), as a comparison with~(*) shows. Similarly, (***) holds even if $m_3=0$ (compare with~(**)), (++) holds even if $q_3=1$ (compare with~(+)), (+++) holds even if $p_3=0$ (compare with~(++)) and~(++++) holds even if $p_3=0$ and $m_3=0$ (compare with~(+++)).

Using the definitions from \eqref{eq:mu23} and combining the results from \eqref{eq:x1x2-p2-geq-2} yields
\begin{subequations}
\begin{align}
  \frac{x_{1,q_3 p_2 \ell_1+p_3 \ell_1-1}+x_{2,1}+x_{2,3}+1}{q_3 p_2 \ell_1+p_3 \ell_1} &\leq 4+\frac{2q_3+\mu_3}{q_3 p_2+p_3} \enspace, \label{eq:ell-5-b21} \\
  \frac{x_{1,\ell_3-1}+x_{2,1}+x_{2,3}+1}{\ell_3} &\leq 4+\frac{1+2q_3-2\mu_3}{q_3 p_2+p_3+\mu_3} \label{eq:ell-5-b22}
\end{align}
\end{subequations}
and
\begin{subequations}
\begin{align}
  \frac{x_{1,p_2 \ell_1-1}+x_{2,2}+x_{2,3}+1}{p_2 \ell_1} &\leq 4+\frac{2(q_3 p_2+p_3)+q_3+\mu_2+\mu_3-1}{p_2} \enspace, \label{eq:ell-6-b21} \\
  \frac{x_{1,\ell_2-1}+x_{2,2}+x_{2,3}+1}{\ell_2} &\leq  4+\frac{2(q_3 p_2+p_3)+q_3-2\mu_2+\mu_3}{p_2+\mu_2} \enspace, \label{eq:ell-6-b22} \\
  \frac{x_{1,\ell_3-1}+x_{2,2}+x_{2,3}+1}{\ell_3} &\leq 4+\frac{2q_3+2p_2+\mu_2-2\mu_3}{q_3 p_2+p_3+\mu_3} \enspace. \label{eq:ell-6-b23} 
\end{align}
\end{subequations}

We first consider the case $c=5$.
If $p_2=p_3=1$ then we obtain
\begin{align}
  \delta(\alpha) &\leByM{\eqref{eq:beta-alpha},\eqref{eq:delta-alpha}}
              \min \{ \eqref{eq:ell-5-b11}, \eqref{eq:ell-5-b12}, \eqref{eq:ell-5-b14}, \eqref{eq:ell-5-b15} \} \notag \\
  &\leByM{\eqref{eq:bounds-delta5-p2-eq-1}} \min \left\{ 5+2\mu_2, 6-\frac{3\mu_2}{1+\mu_2}, 5+\mu_3, 6+\frac{1-4\mu_3}{1+\mu_3} \right\} \enspace. \label{eq:delta5-bound}
\end{align}
In order to determine the best bound resulting from this analysis, we maximize the right hand side of \eqref{eq:delta5-bound} for $\mu_2,\mu_3\in\RR$ with $0\leq \mu_2<1$ and $0\leq \mu_3<1$ (cf.\ the definition in \eqref{eq:mu23} and recall that $m_2,m_3\in\mathbb{N}$ are chosen such that $m_2,m_3<\ell_1$). It is readily checked that the right hand side of \eqref{eq:delta5-bound} attains its maximum for
\begin{equation} \label{eq:delta5-mu23-sol}
  \mu_2=\frac{1}{2}(\sqrt{6}-2) \qquad \text{and} \qquad \mu_3=2\mu_2=\sqrt{6}-2
\end{equation}
(corresponding to the ratios $\ell_2/\ell_1=1+\mu_2=\frac{1}{2}\sqrt{6}$ and $\ell_3/\ell_1=1+\mu_3=\sqrt{6}-1$), yielding
\begin{equation} \label{eq:alpha5-bound-1}
  \delta(\alpha) \leByM{\eqref{eq:delta5-bound},\eqref{eq:delta5-mu23-sol}} \sqrt{6}+3 \enspace.
\end{equation}
If $p_2=1$ and $p_3\geq 2$ then we have
\begin{equation} \label{eq:alpha5-bound-2}
  \delta(\alpha) \leByM{\eqref{eq:beta-alpha},\eqref{eq:delta-alpha},\eqref{eq:ell-5-b13}} 3+\frac{2+2\mu_2}{p_3} < 3+\frac{2+2}{2}=5<\sqrt{6}+3 \enspace.
\end{equation}
For the remaining case $p_2\geq 2$ an easy calculation shows that
\begin{equation} \label{eq:alpha5-bound-3}
  \delta(\alpha) \leByM{\eqref{eq:beta-alpha},\eqref{eq:delta-alpha}} \min \{ \eqref{eq:ell-5-b21}, \eqref{eq:ell-5-b22} \}
  < 5.2 < \sqrt{6}+3 \enspace.
\end{equation}
As the bounds in \eqref{eq:alpha5-bound-1}, \eqref{eq:alpha5-bound-2} and \eqref{eq:alpha5-bound-3} hold for all $\ell\geq 1$ and all $\alpha\in W(\ell,5)$ with $\alpha_{\ell+5-2}=2$ simultaneously, we obtain with the definition in \eqref{eq:delta-c} that
\begin{equation*}
  \delta(5)\leq \sqrt{6}+3 \enspace.
\end{equation*}
To show that this upper bound is tight, by \eqref{eq:delta-c} it suffices to specify a family of strategy sequences $(\alpha^{(t)})_{t\geq 0}$, where $\alpha^{(t)}\in W(\ell_t,5)$ for some $\ell_t\geq 1$ and $\alpha^{(t)}_{\ell_t+5-2}=2$, with $\lim_{t\rightarrow\infty} \delta(\alpha^{(t)})=\sqrt{6}+3$. Define
\begin{equation*}
  \alpha^{(t)}:=(1)^{\ell_{1,t}-1}\circ(2)\circ(1)^{\ell_{2,t}-\ell_{1,t}}\circ(2)\circ(1)^{\ell_{3,t}-\ell_{2,t}}\circ(2,2) \in W(\ell_{3,t},5) \enspace,
\end{equation*}
for all $t\geq 0$, where $\ell_{1,t}:=10^t$, $\ell_{2,t}:=\lfloor \frac{1}{2}\sqrt{6}\cdot 10^t\rfloor$ and $\ell_{3,t}:=\lfloor (\sqrt{6}-1)\cdot 10^t\rfloor$, i.e., we have
\begin{equation} \label{eq:delta5-lim-ell-ratio}
  \lim_{t\rightarrow\infty} \frac{\ell_{2,t}}{\ell_{1,t}} = \frac{1}{2}\sqrt{6} \qquad \text{and} \qquad
  \lim_{t\rightarrow\infty} \frac{\ell_{3,t}}{\ell_{1,t}} = \sqrt{6}-1\enspace.
\end{equation}
For each such strategy sequence $\alpha^{(t)}$ we obtain from \eqref{eq:recursion}, using that $\ell_{2,t},\ell_{3,t}<2\ell_{1,t}$,
\begin{equation} \label{eq:delta5-x1x2-alpha-t}
\begin{split}
  x_{1,j} &= j \enspace, \quad 0\leq j\leq \ell_{1,t}-1 \enspace, \\
  x_{2,1} &= \ell_{1,t} \enspace, \\
  x_{1,j} &= \ell_{1,t}+j \enspace, \quad \ell_{1,t}\leq j\leq \ell_{2,t}-1 \enspace, \\
  x_{2,2} &= \ell_{1,t}+\ell_{2,t} \enspace, \\
  x_{1,j} &= 2\ell_{1,t}+j \enspace, \quad \ell_{2,t}\leq j\leq \ell_{3,t}-1 \enspace, \\
  x_{2,3} &= 2\ell_{1,t}+\ell_{3,t} \enspace, \\
  x_{2,4} &= 2\ell_{1,t}+\ell_{2,t}+\ell_{3,t} \enspace, \\
\end{split}
\end{equation}
implying that
\begin{equation} \label{eq:delta5-beta-alpha-t}
  \beta(\alpha^{(t)}) \eqByM{\eqref{eq:beta-alpha},\eqref{eq:delta5-x1x2-alpha-t}} 2\ell_{1,t}+2\ell_{2,t}+1
\end{equation}
and
\begin{equation} \label{eq:delta5-delta-alpha-t}
  \delta(\alpha^{(t)}) \eqByM{\eqref{eq:delta-alpha},\eqref{eq:delta5-x1x2-alpha-t},\eqref{eq:delta5-beta-alpha-t}}
   \min \left\{ 3+2\frac{\ell_{2,t}}{\ell_{1,t}} ,
                3+3\Big(\frac{\ell_{2,t}}{\ell_{1,t}}\Big)^{-1} ,
                1+\Big(4+2\frac{\ell_{2,t}}{\ell_{1,t}}\Big)\Big(\frac{\ell_{3,t}}{\ell_{1,t}}\Big)^{-1}
        \right\} \enspace.
\end{equation}
Using \eqref{eq:delta5-lim-ell-ratio} it follows from \eqref{eq:delta5-delta-alpha-t} that $\delta(\alpha^{(t)})\rightarrow \sqrt{6}+3$ as $t\rightarrow\infty$.

We now consider the case $c=6$.
If $p_2=p_3=1$ then we obtain
\begin{align}
  \delta(\alpha) &\leByM{\eqref{eq:beta-alpha},\eqref{eq:delta-alpha}}
              \min \{ \eqref{eq:ell-6-b12}, \eqref{eq:ell-6-b13}, \eqref{eq:ell-6-b14} \} \notag \\
  &\leByM{\eqref{eq:bounds-delta6-p2-eq-1}} \min \left\{ 4+\frac{3-2\mu_2+\mu_3}{1+\mu_2}, 6+\mu_2+\mu_3, 6+\frac{2+\mu_2-4\mu_3}{1+\mu_3} \right\} \enspace. \label{eq:delta6-bound}
\end{align}
Maximizing the right hand side of \eqref{eq:delta6-bound} for $\mu_2,\mu_3\in\mathbb{R}$ with $0\leq \mu_2<1$ and $0\leq \mu_3<1$, it is readily checked that the maximum is attained for
\begin{equation} \label{eq:delta6-mu23-sol}
  \mu_2=\frac{1}{6}(\sqrt{37}-5) \qquad \text{and} \qquad \mu_3=2\mu_2=\frac{1}{3}(\sqrt{37}-5)
\end{equation}
(corresponding to the ratios $\ell_2/\ell_1=1+\mu_2=\frac{1}{6}(\sqrt{37}+1)$ and $\ell_3/\ell_1=1+\mu_3=\frac{1}{3}(\sqrt{37}-2)$), yielding
\begin{equation} \label{eq:alpha6-bound-1}
  \delta(\alpha) \leByM{\eqref{eq:delta6-bound},\eqref{eq:delta6-mu23-sol}} \frac{1}{2}(\sqrt{37}+7) \enspace.
\end{equation}
If $p_2=1$ and $p_3\geq 2$ then an easy calculation shows that
\begin{equation} \label{eq:alpha6-bound-2}
  \delta(\alpha) \leByM{\eqref{eq:beta-alpha},\eqref{eq:delta-alpha}} \min \{ \eqref{eq:ell-6-b13}, \eqref{eq:ell-6-b14} \} < 6.4 < \frac{1}{2}(\sqrt{37}+7) \enspace.
\end{equation}
For the remaining case $p_2\geq 2$ it is elementary but somewhat tedious to check that
\begin{equation} \label{eq:alpha6-bound-3}
  \delta(\alpha) \leByM{\eqref{eq:beta-alpha},\eqref{eq:delta-alpha}} \min \left\{ \eqref{eq:ell-6-b21}, \eqref{eq:ell-6-b22}, \eqref{eq:ell-6-b23} \right\}
  < 6.4 < \frac{1}{2}(\sqrt{37}+7) \enspace.
\end{equation}
As the bounds in \eqref{eq:alpha6-bound-1}, \eqref{eq:alpha6-bound-2} and \eqref{eq:alpha6-bound-3} hold for all $\ell\geq 1$ and all $\alpha\in W(\ell,6)$ with $\alpha_{\ell+6-2}=2$ simultaneously, we obtain with the definition in \eqref{eq:delta-c} that
\begin{equation*}
  \delta(6) \leq \frac{1}{2}(\sqrt{37}+7) \enspace.
\end{equation*}
To see that this upper bound is tight, consider the family $(\alpha^{(t)})_{t\geq 0}$ of strategy sequences
\begin{equation*}
  \alpha^{(t)}:=(1)^{\ell_{1,t}-1}\circ(2)\circ(1)^{\ell_{2,t}-\ell_{1,t}}\circ(2)\circ(1)^{\ell_{3,t}-\ell_{2,t}}\circ(2,2,2) \in W(\ell_{3,t},6)
\end{equation*}
with $\ell_{1,t}:=10^t$, $\ell_{2,t}:=\lfloor \frac{1}{6}(\sqrt{37}+1)\cdot 10^t\rfloor$ and $\ell_{3,t}:=\lfloor \frac{1}{3}(\sqrt{37}-2)\cdot 10^t\rfloor$, i.e., we have
\begin{equation} \label{eq:delta6-lim-ell-ratio}
  \lim_{t\rightarrow\infty} \frac{\ell_{2,t}}{\ell_{1,t}} = \frac{1}{6}(\sqrt{37}+1) \qquad \text{and} \qquad
  \lim_{t\rightarrow\infty} \frac{\ell_{3,t}}{\ell_{1,t}} = \frac{1}{3}(\sqrt{37}-2)\enspace.
\end{equation}
For each such strategy sequence $\alpha^{(t)}$ the sequences $x_1$ and $x_2$ from the recursion~\eqref{eq:recursion} satisfy the relations in \eqref{eq:delta5-x1x2-alpha-t}, and in addition
\begin{equation} \label{eq:delta6-x25-alpha-t}
  x_{2,5} = 2\ell_{1,t}+2\ell_{2,t}+\ell_{3,t} \enspace, \\
\end{equation}
implying that
\begin{equation} \label{eq:delta6-beta-alpha-t}
  \beta(\alpha^{(t)}) \eqByM{\eqref{eq:beta-alpha},\eqref{eq:delta5-x1x2-alpha-t},\eqref{eq:delta6-x25-alpha-t}} 3\ell_{1,t}+\ell_{2,t}+\ell_{3,t}+1
\end{equation}
and
\begin{equation} \label{eq:delta6-delta-alpha-t}
  \delta(\alpha^{(t)}) \eqByM{\eqref{eq:delta-alpha},\eqref{eq:delta5-x1x2-alpha-t},\eqref{eq:delta6-x25-alpha-t},\eqref{eq:delta6-beta-alpha-t}}
   \min \left\{ 4+\frac{\ell_{2,t}}{\ell_{1,t}}+\frac{\ell_{3,t}}{\ell_{1,t}} ,
                2+\Big(4+\frac{\ell_{3,t}}{\ell_{1,t}}\Big)\Big(\frac{\ell_{2,t}}{\ell_{1,t}}\Big)^{-1} ,
                2+\Big(5+\frac{\ell_{2,t}}{\ell_{1,t}}\Big)\Big(\frac{\ell_{3,t}}{\ell_{1,t}}\Big)^{-1}
        \right\} \enspace.
\end{equation}
Using \eqref{eq:delta6-lim-ell-ratio} it follows from \eqref{eq:delta6-delta-alpha-t} that $\delta(\alpha^{(t)})\rightarrow \frac{1}{2}(\sqrt{37}+7)$ as $t\rightarrow\infty$.
\end{proof}

\end{document}